\documentclass[english,american]{article}
\usepackage[T1]{fontenc}
\usepackage[utf8]{inputenc}
\usepackage[a4paper]{geometry}
\usepackage{verbatim}
\usepackage{longtable}
\usepackage{float}
\usepackage{amsthm}
\usepackage{amsmath}
\usepackage{amssymb}
\usepackage{graphicx}

\makeatletter

\providecommand{\tabularnewline}{\\}

\numberwithin{equation}{section}
\numberwithin{figure}{section}
\usepackage{enumitem}		
\newenvironment{lyxcode}
{\par\begin{list}{}{
\setlength{\rightmargin}{\leftmargin}
\setlength{\listparindent}{0pt}
\raggedright
\setlength{\itemsep}{0pt}
\setlength{\parsep}{0pt}
\normalfont\ttfamily}%
 \item[]}
{\end{list}}
  \theoremstyle{plain}
  \newtheorem*{prop*}{\protect\propositionname}
  \theoremstyle{remark}
  \newtheorem*{conclusion*}{\protect\conclusionname}
  \theoremstyle{remark}
  \newtheorem*{notation*}{\protect\notationname}
\theoremstyle{plain}
\newtheorem{thm}{\protect\theoremname}
  \theoremstyle{definition}
  \newtheorem{condition}[thm]{\protect\conditionname}
  \theoremstyle{remark}
  \newtheorem*{note*}{\protect\notename}
  \theoremstyle{remark}
  \newtheorem{claim}[thm]{\protect\claimname}
  \theoremstyle{definition}
  \newtheorem{defn}[thm]{\protect\definitionname}
  \theoremstyle{plain}
  \newtheorem{cor}[thm]{\protect\corollaryname}
  \theoremstyle{plain}
  \newtheorem{fact}[thm]{\protect\factname}
  \theoremstyle{remark}
  \newtheorem{notation}[thm]{\protect\notationname}
  \theoremstyle{remark}
  \newtheorem{note}[thm]{\protect\notename}
  \theoremstyle{plain}
  \newtheorem{prop}[thm]{\protect\propositionname}
  \theoremstyle{plain}
  \newtheorem{question}[thm]{\protect\questionname}
  \theoremstyle{plain}
  \newtheorem{conjecture}[thm]{\protect\conjecturename}
  \theoremstyle{plain}
  \newtheorem{assumption}[thm]{\protect\assumptionname}
  \theoremstyle{remark}
  \newtheorem{conclusion}[thm]{\protect\conclusionname}
  \theoremstyle{plain}
  \newtheorem*{fact*}{\protect\factname}
  \theoremstyle{plain}
  \newtheorem{lem}[thm]{\protect\lemmaname}
 \theoremstyle{definition}
 \newtheorem*{defn*}{\protect\definitionname}
  \theoremstyle{plain}
  \newtheorem*{lem*}{\protect\lemmaname}
  \theoremstyle{plain}
  \newtheorem*{thm*}{\protect\theoremname}
  \theoremstyle{plain}
  \newtheorem*{cor*}{\protect\corollaryname}
  \theoremstyle{plain}
  \newtheorem{algorithm}[thm]{\protect\algorithmname}

\usepackage{algorithm}
\numberwithin{thm}{section} 

\@ifundefined{showcaptionsetup}{}{%
 \PassOptionsToPackage{caption=false}{subfig}}
\usepackage{subfig}
\makeatother

\usepackage{babel}
  \addto\captionsamerican{\renewcommand{\algorithmname}{\inputencoding{latin9}Algorithm}}
  \addto\captionsamerican{\renewcommand{\assumptionname}{\inputencoding{latin9}Assumption}}
  \addto\captionsamerican{\renewcommand{\claimname}{\inputencoding{latin9}Claim}}
  \addto\captionsamerican{\renewcommand{\conclusionname}{\inputencoding{latin9}Conclusion}}
  \addto\captionsamerican{\renewcommand{\conditionname}{\inputencoding{latin9}Condition}}
  \addto\captionsamerican{\renewcommand{\conjecturename}{\inputencoding{latin9}Conjecture}}
  \addto\captionsamerican{\renewcommand{\corollaryname}{\inputencoding{latin9}Corollary}}
  \addto\captionsamerican{\renewcommand{\definitionname}{\inputencoding{latin9}Definition}}
  \addto\captionsamerican{\renewcommand{\factname}{\inputencoding{latin9}Fact}}
  \addto\captionsamerican{\renewcommand{\lemmaname}{\inputencoding{latin9}Lemma}}
  \addto\captionsamerican{\renewcommand{\notationname}{\inputencoding{latin9}Notation}}
  \addto\captionsamerican{\renewcommand{\notename}{\inputencoding{latin9}Note}}
  \addto\captionsamerican{\renewcommand{\propositionname}{\inputencoding{latin9}Proposition}}
  \addto\captionsamerican{\renewcommand{\questionname}{\inputencoding{latin9}Question}}
  \addto\captionsamerican{\renewcommand{\theoremname}{\inputencoding{latin9}Theorem}}
  \addto\captionsenglish{\renewcommand{\algorithmname}{\inputencoding{latin9}Algorithm}}
  \addto\captionsenglish{\renewcommand{\assumptionname}{\inputencoding{latin9}Assumption}}
  \addto\captionsenglish{\renewcommand{\claimname}{\inputencoding{latin9}Claim}}
  \addto\captionsenglish{\renewcommand{\conclusionname}{\inputencoding{latin9}Conclusion}}
  \addto\captionsenglish{\renewcommand{\conditionname}{\inputencoding{latin9}Condition}}
  \addto\captionsenglish{\renewcommand{\conjecturename}{\inputencoding{latin9}Conjecture}}
  \addto\captionsenglish{\renewcommand{\corollaryname}{\inputencoding{latin9}Corollary}}
  \addto\captionsenglish{\renewcommand{\definitionname}{\inputencoding{latin9}Definition}}
  \addto\captionsenglish{\renewcommand{\factname}{\inputencoding{latin9}Fact}}
  \addto\captionsenglish{\renewcommand{\lemmaname}{\inputencoding{latin9}Lemma}}
  \addto\captionsenglish{\renewcommand{\notationname}{\inputencoding{latin9}Notation}}
  \addto\captionsenglish{\renewcommand{\notename}{\inputencoding{latin9}Note}}
  \addto\captionsenglish{\renewcommand{\propositionname}{\inputencoding{latin9}Proposition}}
  \addto\captionsenglish{\renewcommand{\questionname}{\inputencoding{latin9}Question}}
  \addto\captionsenglish{\renewcommand{\theoremname}{\inputencoding{latin9}Theorem}}
  \providecommand{\algorithmname}{\inputencoding{latin9}Algorithm}
  \providecommand{\assumptionname}{\inputencoding{latin9}Assumption}
  \providecommand{\claimname}{\inputencoding{latin9}Claim}
  \providecommand{\conclusionname}{\inputencoding{latin9}Conclusion}
  \providecommand{\conditionname}{\inputencoding{latin9}Condition}
  \providecommand{\conjecturename}{\inputencoding{latin9}Conjecture}
  \providecommand{\corollaryname}{\inputencoding{latin9}Corollary}
  \providecommand{\definitionname}{\inputencoding{latin9}Definition}
  \providecommand{\factname}{\inputencoding{latin9}Fact}
  \providecommand{\lemmaname}{\inputencoding{latin9}Lemma}
  \providecommand{\notationname}{\inputencoding{latin9}Notation}
  \providecommand{\notename}{\inputencoding{latin9}Note}
  \providecommand{\propositionname}{\inputencoding{latin9}Proposition}
  \providecommand{\questionname}{\inputencoding{latin9}Question}
  \providecommand{\theoremname}{\inputencoding{latin9}Theorem}
\providecommand{\theoremname}{\inputencoding{latin9}Theorem}

\begin{document}

\title{Preventing Exceptions to Robins InEquality}

\author{Thomas Schwabhäuser%
\thanks{schwabts@googlemail.com%
}}

\date{}
\maketitle
\begin{abstract}
\noindent For sufficiently large $n$ Ramanujan gave a sufficient
condition for the truth Robin's InEquality $X(n):=\frac{\sigma(n)}{n\ln\ln n}<e^{\gamma}$
(RIE). The largest known violation of RIE is $n_{8}=5040$. In this
paper Robin's multipliers are split into logarithmic terms $\mathcal{L}$
and relative divisor sums $\mathcal{G}$. A violation of RIE above
$n_{8}$ is proposed to imply oscillations that cause $\mathcal{G}$
to exceed $\mathcal{L}$. To this aim Alaoglu and Erdős's conjecture
for the CA numbers algorithm is used and the paper could almost be
reduced to section~\ref{sub:Oscillation-Theorems} on pages~\pageref{sub:Oscillation-Theorems}
to~\pageref{sec:Final-Remarks}.\end{abstract}
\begin{lyxcode}
\tableofcontents{}

\newpage{}
\end{lyxcode}

\section{Introduction}

\subsection{Outline}

Robin's Inequality $\frac{\sigma(n)}{n\ln\ln n}<e^{\gamma}$ (RIE)
for sufficiently large $n$ can be derived from Ramanujan's Lost Notebook
as necessary condition for RH. Unfortunately his work was not published
until 1997. The inequality can be derived from an asymptotic expression
that emerged from the study of generalised highly composite and generalised
superior highly composite numbers. Alaoglu and Erdős coined the terms
superabundant (SA) and colossally abundant (CA) in 1944 and mentioned
the role of transcendental number theory in the process of finding
CA numbers.
\begin{prop*}
\label{prop:Ramanujans-RIE-under-RH}(Ramanujan \cite[(382)]{Ramanujan:1997})

If RH is true, \cite[§56]{Ramanujan:1997}, it follows that

\[
{\textstyle \sum_{-1}}\left(N\right)=e^{\gamma}\left(\ln\ln N-\frac{2\left(\sqrt{2}-1\right)}{\sqrt{\ln n}}+S_{1}\left(\ln N\right)+\frac{O\left(1\right)}{\sqrt{\ln N}\ln\ln N}\right).
\]
\end{prop*}
\begin{conclusion*}
There is an $n_{0}$ such that $\frac{\sigma\left(n\right)}{n}<e^{\gamma}\ln\ln n$
for all $n>n_{0}$.

{[}Rf. notes at the end of \cite{Ramanujan:1997}.{]}
\end{conclusion*}
Robin clarified the meaning of ``sufficiently large'' in 1984 by
finding
\begin{enumerate}
\item that the function $X\left(n\right):=\frac{\sigma(n)}{n\ln\ln n}$
takes maximal values on CA numbers.
\item It is sufficient for RH that RIE holds true for sufficiently large
$n$, i.e. for $n>5040$.
\item The oscillation theorem $X\left(n\right)=e^{\gamma}\cdot\left(1+\Omega_{\pm}\left(\left(\ln n\right)^{-b}\right)\right)$
in CA numbers if RH is false.
\item $X\left(n\right)$ has an unconditional bound $B\left(n\right)$ for
some $B\left(n\right)=e^{\gamma}+o\left(1\right)$.
\end{enumerate}
The major tool were estimations with Chebyshev's functions $\psi$
and $\vartheta$ using the results of Rosser and Schoenfeld. In order
to show that $X$ takes maximal values on CA numbers Robin used a
multiplier consisting of ratios of relative divisor sums $\sigma_{-1}(n)$
of consecutive CA numbers and iterated logarithms. The argument is
iterated in this report which proves that $X$ takes a greater value
on a subsequent CA number if the product of ratios of relative divisor
sums of intermediate CA numbers exceeds the respective product of
ratios of iterated logarithms. But the CA numbers algorithm relies
on the quotient of consecutive CA numbers to be prime which is not
guaranteed unless Alaoglu and Erdős' special case of the Four Exponentials
Conjecture is true.

The point of this investigation has been to find out if the minimal
oscillations in case RH is false will force the products of ratios
of relative divisor sums to become greater than the corresponding
products of ratios of iterated logarithms. This has been achieved
by finding a template for the quotient of maximal and minimal values
of $X\left(n\right)$ as $n$ proceeds in CA numbers and analysing
the template with polar coordinates.

The paper is primarily organised as a chain of reductions that is
summarised in the final Conclusion~\ref{conclusion:reduction-of-RIE}.
Section~\ref{sub:Preparation} establishes the need to find multiples
of every natural $n$ on which $X$ takes a greater value than it
takes on $n$. Such multiples prevent $n$ from being an exception.
Section~\ref{sec:Colosally-Abundant-Numbers} demonstrates how the
multipliers used by Robin work and how they can be split. This method
is iterated in Section~\ref{sec:Subsequent Maximisers} to show the
sufficiency of testing $\mathcal{G}>\mathcal{L}$ for $\mathcal{G}=\frac{\sigma_{-1}(nx)}{\sigma_{-1}(n)}$
and $\mathcal{L}=\frac{\ln\ln nx}{\ln\ln n}$. Similar conditions
were found in \cite{Morkotun:2013,Nazardonyavi:Yakubovich:2013,Nazardonyavi:Yakubovich:2013:Delicacy}
during the course of my investigation. The setup of the latter two
reports is summarised using the present setup as a part of section~\ref{sec:The-Question-of-Life}
after presenting some numerical data. Then Mertens' theorem motivates
expecting the truth of $\mathcal{G}>\mathcal{L}$ before Robin's oscillation
theorem is used in section~\ref{sub:Oscillation-Theorems} to propose
an indirect proof.

\subsection{Preparation\label{sub:Preparation}}
\begin{notation*}
Let $X(n):=\frac{\sigma(n)}{n\ln\ln n}$ with the sum of divisors
$\sigma$, write $\textrm{RIE}\left(n\right)$ short for Robin's InEquality
$X(n)<e^{\gamma}$,\cite{Robin:1984:grandes:valeurs}, and denote
the set of primes $\left\{ p_{n}\right\} _{n=1}^{\infty}=2,3,5,\ldots$
by $\mathbb{P}$. The $k$th largest prime factor of an integer $n$
is denoted by $P_{k}\left(n\right)$, \cite[5.17]{Riesel:2011}. Also
let $\left[a,b\right]_{\mathbb{N}}:=\left[a,b\right]\cap\mathbb{N}$.
\end{notation*}
Grönwall \cite{Gronwall:1913} mentioned that the asymptotic behaviour
of the function $Y(n):=\frac{\varphi(n)}{n}\cdot\ln\ln n$ had been
studied by Landau, \cite{Landau:1909}. Then he proved Theorem \ref{thm:Gr=0000F6nwall}
below. Rf.\foreignlanguage{english}{ \cite{Nicolas:1983,Choie:2007}
for Nicolas' inequality and }\cite{Planat:Sole:2011:a,Planat:Sole:2011c}
for approaches with the\foreignlanguage{english}{ Dedekind $\psi$
function. }Suppose
\begin{condition}
\label{condition:sufficient}For every $n>5040$ there is a number
$x$ such that $X(nx)>X(n)$.\end{condition}
\begin{note*}
This section establishes\end{note*}
\begin{claim}
\label{claim:general-validity-of-RIE}$\textrm{RIE}\left(n\right)$
is true for all $n>5040$.
\end{claim}
If the opposite of Condition~\ref{condition:sufficient} was true
for some $n>5040$ the number $n$ may be said to be \textbf{\textit{exceptional}}
since no such $n$ is known so far. Without requiring $n>5040$ this
is called GA2 in \cite[p. 2]{Caveney:2011}. Known GA2 numbers are
3, 4, 5, 6, 8, 10, 12, 18, 24, 36, 48, 60, 72, 120, 180, 240, 360,
2520, and 5040. Recall
\begin{thm}
\label{thm:Gr=0000F6nwall}(Grönwall)\hspace{8em}$\limsup\limits _{n\to\infty}X(n)=e^{\gamma}.$$ $
\end{thm}
This is easily extended.
\begin{thm}
\label{thm:extended-Gr=0000F6nwall}\hspace{10em}$\limsup\limits _{x\to\infty}X(n\cdot x)=e^{\gamma}=1.78107\,24179\,90197....$\end{thm}
\begin{proof}
An adaption of \cite[§22.9]{Hardy:Wright:1979}. Rf. \cite[App. A]{Knuth:1997}
for the numerical value.\end{proof}
\begin{defn}
$\ensuremath{C:=\left(\frac{7}{3}-e^{\gamma}\cdot\ln\ln12\right)\cdot\ln\ln12\approx0.64821365}$,
rf. \cite[Theorem 1.1]{Briggs:2006}, \cite[Eq. (1.4)]{Nazardonyavi:Yakubovich:2013},
or \cite[Lemma 13]{Rosser:Schoenfeld:1962} and \cite[Theorem 7]{Schoenfeld:1976}.\end{defn}
\begin{thm}
\label{thm:Unconditional_Bound}(\cite[Théorème 2]{Robin:1984:grandes:valeurs})

$X(n)\le B\left(n\right):=e^{\gamma}+C\cdot\left(\ln\ln n\right)^{-2}$
for all $n\in\mathbb{N}\setminus\left\{ 1,2,12\right\} $.
\end{thm}
Thus, assuming Condition~\ref{condition:sufficient} it is easily
seen that a minimal counterexample of RIE above 5040 contradicts Theorem~\ref{thm:Unconditional_Bound}
since for any number $n$ Condition~\ref{condition:sufficient} implies
the existence of a non-decreasing sequence of values of $X$ that
starts at $X(n)$. If $X(n)>e^{\gamma}$ choose $x_{1}\cdot\cdots\cdot x_{k}$
such that 
\begin{enumerate}
\item $X(n\cdot y_{i+1})>X(n\cdot y_{i})$ for all $i\in\left[0,k-1\right]_{\mathbb{N}}$
where $y_{i}:=\prod\left\{ x_{j};j\in\left[1,i\right]_{\mathbb{N}}\right\} $
and
\item $C\cdot\ln\ln\left(n\cdot y_{k}\right)^{-2}<X(n)-e^{\gamma}$, i.e.
$n\cdot y_{k}>\exp\left(\exp\left(\sqrt{\frac{C}{X(n)-e^{\gamma}}}\right)\right)=B^{-1}\left(X(n)\right)$.
\end{enumerate}
The contradiction $X(n\cdot y_{k})>e^{\gamma}+C\cdot\ln\ln(n\cdot y_{k})^{-2}$
to Theorem~\ref{thm:Unconditional_Bound} follows. Thus
\begin{thm}
\label{thm:sufficient}Claim~\ref{claim:general-validity-of-RIE}
follows from Condition~\ref{condition:sufficient}.
\end{thm}
Assuming Condition~\ref{condition:sufficient} a consequence of \cite[Thm 5]{Caveney:2011}
is
\begin{cor}
\label{thm:GA2-numbers}

There is no GA2 number $n>5040$. In other words there is no exceptional
number.
\end{cor}
Proving the absence of exceptional numbers seems to be just as difficult
as proving Condition~\ref{condition:sufficient}. This is no surprise
because a one is an indirect proof of the other.

\section{Colosally Abundant Numbers\label{sec:Colosally-Abundant-Numbers}}
\begin{defn}
Rf. \cite[Superabundant and colossally abundant number]{Wikipedia}
\begin{enumerate}
\item $n$ is SA if it meets $\sigma_{-1}(n)\geq\sigma_{-1}(k)$ for all
$k<n$, rf. \cite[p. 839]{Weisstein:1999}, \cite[A004394]{OEIS}.
\item $n$ is CA if $\sigma_{-1}(n)\cdot n^{-\varepsilon}\geq\sigma_{-1}(k)\cdot k^{-\varepsilon}$
for all $k$ and an $\varepsilon>0$, rf. \cite[A004490]{OEIS}.
\end{enumerate}
\end{defn}
By Theorem~\ref{thm:Gr=0000F6nwall} there are infinitely many SA
numbers but they are only mentioned here because the SA property suffices
to determine the asymptotic behaviour of $P_{1}\left(n\right)$.
\begin{fact}
\label{fact:basic-CA-number-properties}\ 
\begin{enumerate}
\item By \cite[p. 68]{Erdoes:Nicolas:1975} CA numbers are SA.
\item If $n$ is SA and $p$ the largest prime factor in $n$ then $p\sim\ln n$
by \cite[Theorem 7]{Alaoglu:Erdoes:1944}.
\end{enumerate}
\end{fact}
Rf. \cite{Noe:2005,Noe:2009} for more information.
\begin{quotation}
``A superparticular number is when a great number contains a lesser
number, to which it is compared, and at the same time one part of
it.'' Rf. \cite[p.III.6.12,n.7]{Throop:IsidoreEtymologiesI:2006}.\end{quotation}
\begin{notation}
\label{notation:setup-for-CA-numbers}Let $F\left(x,v\right):=\frac{1}{\ln x}\ln\left(G\left(x,v\right)\right)$
with $G\left(x,v\right):=1+\frac{1}{g\left(x,v\right)}$; $g\left(x,v\right):=\sigma\left(x^{v}\right)-1$.
In virtue of assumption~\ref{assumption:Alaoglu-Erdoes-conjecture}
below the parameters $\varepsilon$ of CA numbers belong to the set
$\mathcal{E}:=\bigcup\limits _{p\in\mathbb{P}}\left\{ F\left(p,v\right);v\in\mathbb{N}\right\} $.
Write $E$ as decreasing sequence $\mathcal{E}=\left(\varepsilon_{i}\right)_{i=1}^{\infty}$
allowing for the definition $n_{i}:=n\left(\varepsilon_{i}\right)$,
rf. \cite{Ramanujan:1997,Alaoglu:Erdoes:1944,Erdoes:Nicolas:1975,Briggs:2006}.
Additionally put $q_{i+1}:=\frac{n_{i+1}}{n_{i}}$ and $g_{i}:=G\left(q_{i},v_{q_{i}}\left(n_{i}\right)\right)$
such that $\varepsilon_{i}=F\left(q_{i},v_{q_{i}}\left(n_{i}\right)\right)=\log_{q_{i}}\left(g_{i}\right)$
and $q_{i}^{\varepsilon_{i}}=g_{i}=\frac{\sigma_{-1}\left(n_{i}\right)}{\sigma_{-1}\left(n_{i-1}\right)}$
is superparticular. Let $q_{0}:=1=:n_{0}$, $\varepsilon_{0}=1$,
and $L_{i}\left(n\right):=\frac{\ln\left(\ln\left(n\right)\right)}{\ln\left(\ln\left(n_{i}\right)\right)}$,
too.\end{notation}
\begin{defn}
Let $f_{\varepsilon}$ denote the function $x\ {\mapsto}\ \varepsilon\cdot x-\ln\left(\ln\left(x\right)\right)$
and \emph{Robin's multiplier} be $R_{j}\left(n,\varepsilon\right):=\exp\left(f_{\varepsilon}\left(\ln n\right)-f_{\varepsilon}\left(\ln n_{j}\right)\right)=\left(\frac{n}{n_{j}}\right)^{\varepsilon}\cdot\left(L_{j}\left(n\right)\right)^{-1}$.\end{defn}
\begin{note}
\label{note:convexity}
\begin{enumerate}
\item The derivatives of $f_{\varepsilon}$ are given by 
\[
f_{\varepsilon}'(x)=\varepsilon-\frac{1}{x\cdot\ln\left(x\right)}\quad\text{{and}}\quad f_{\varepsilon}''(x)=\frac{1}{x^{2}\ln\left(x\right)}+\frac{1}{\left(x\cdot\ln\left(x\right)\right)^{2}}>0
\]
which implies that $f_{\varepsilon}\in C^{2}\left(1,\infty\right)$
is convex for every $\varepsilon>0$.
\item $f_{\varepsilon}\left(x\right)\rightarrow\infty$ as $x\rightarrow1$
or as $x\rightarrow\infty$ since logs grow slower than any power
of $x$.
\item If $\varepsilon_{1}<\varepsilon_{2}$ then $f_{\varepsilon_{1}}\left(x\right)<f_{\varepsilon_{2}}\left(x\right)$
for all $x$ because $f_{\varepsilon_{2}}\left(x\right)-f_{\varepsilon_{1}}\left(x\right)=\left(\varepsilon_{2}-\varepsilon_{1}\right)\cdot x>0$.
In particular $f_{\varepsilon_{1}}$ and $f_{\varepsilon_{2}}$ do
not intersect if $\varepsilon_{1}\ne\varepsilon_{2}$.
\item $R_{i}\left(n_{i+1},\varepsilon_{i+1}\right)\cdot L_{i}\left(n_{i+1}\right)=g_{i+1}$
and $X\left(n_{i}\right)\cdot R_{i}\left(n_{i+1},\varepsilon_{i+1}\right)=X\left(n_{i+1}\right)$.
\item $\varepsilon=\varepsilon_{i+1}$ if $\sigma_{-1}(n_{i})\cdot n_{i}^{-\varepsilon}=\sigma_{-1}(n_{i+1})\cdot n_{i+1}^{-\varepsilon}$.
\end{enumerate}
\end{note}
\begin{prop}
\label{prop:maximising-CA-numbers}(\cite[§3, Prop 1]{Robin:1984:grandes:valeurs},
\cite[p. 237]{Robin:1983:ordre:maximum})

Colosally abundant numbers maximise $X$, i.e. $X(n)\le\max\left(X(n_{i}),X(n_{i+1})\right)$
if $n\in\left[n_{i},n_{i+1}\right]_{\mathbb{N}}$.\end{prop}
\begin{proof}
$X(n)\le X(n_{j})$ follows from $R_{j}\left(n,\varepsilon\right)\leq1$
since $X\left(n\right)\leq X\left(n_{j}\right)R_{j}\left(n,\varepsilon\right)$
follows from $\sigma_{-1}(n)\cdot n^{-\varepsilon}\leq\sigma_{-1}(n_{j})\cdot n_{j}^{-\varepsilon}$
with Note~\ref{note:convexity}, \#4/5 if $n\in\left[n_{i},n_{i+1}\right]_{\mathbb{N}}$,
$\varepsilon=\varepsilon_{i+1}$, and $j\in\left\{ i,i+1\right\} $.
For $\frac{n^{\varepsilon}}{\ln\ln n}\le\frac{n_{j}^{\varepsilon}}{\ln\ln n_{j}}$
it is sufficient to show the consequence $f_{\varepsilon}\left(\ln n\right)\leq\max\left(f_{\varepsilon}\left(\ln n_{i}\right),f_{\varepsilon}\left(\ln n_{i+1}\right)\right)$
of Note~\ref{note:convexity}, \#1.
\end{proof}

\noindent
\begin{figure}
\includegraphics[width=7cm]{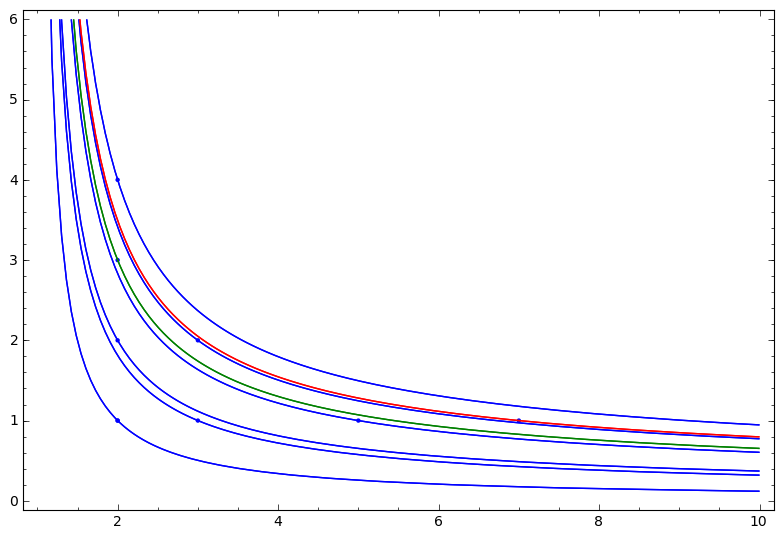}\includegraphics[width=7cm]{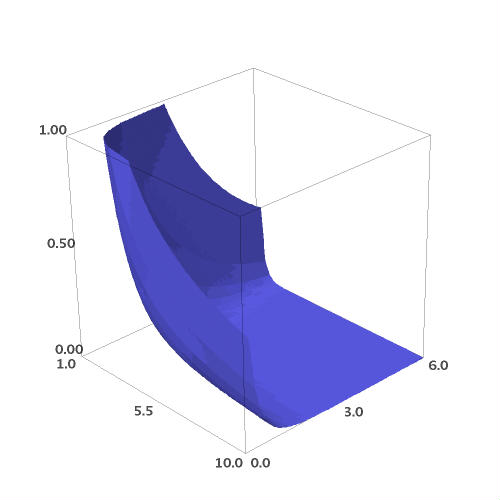}

\caption{\label{fig:Surface-Eps-Eq-F}The surface $\varepsilon=F\left(x,v\right)$}
\end{figure}

\vspace{1ex}

Based on this setup the algorithm computing the sequence $\left(n_{i}\right)_{i}$
of CA numbers seems to be well-understood, \cite{Ramanujan:1915,Alaoglu:Erdoes:1944,Erdoes:Nicolas:1975,Robin:1984:grandes:valeurs,Ramanujan:1997,Noe:2005,Briggs:2006,Noe:2009,Choie:2007,Caveney:2011,Caveney:2012,OEIS}.
Nevertheless there is the open
\begin{question}
\ 
\begin{enumerate}
\item Can the algorithm find a new violation of RIE? (Does one exist?)
\item Do two consecutive CA numbers exist whose quotient $q_{i}$ is semiprime?
\end{enumerate}
\end{question}
Item 2. has been given a negative answer under the following Conjecture,
rf. \cite[p. 455]{Alaoglu:Erdoes:1944}.
\begin{conjecture}
\label{conjecture:Four-Exponentials}(Four Exponentials)

If $x_{1}$, $x_{2}$ and $y_{1}$, $y_{2}$ are two pairs of complex
numbers, with each pair being linearly independent over the rational
numbers, then at least one of the following four numbers is transcendental:
\[
e^{x_{1}y_{1}},e^{x_{1}y_{2}},e^{x_{2}y_{1}},e^{x_{2}y_{2}}\,,
\]
rf. \cite[p. 71]{Erdoes:Nicolas:1975}, \cite[ch. 2]{Lang:1966},
\cite[ch. 2]{Waldschmidt:1974}, \cite[Section 1.3]{Waldschmidt:2000}.
\end{conjecture}
Semiprime quotients cause unexpected difficulties. Therefore I assume
a special case of Conjecture~\ref{conjecture:Four-Exponentials}.
\begin{assumption}
\label{assumption:Alaoglu-Erdoes-conjecture}(Alaoglu and Erdős)

For any two distinct prime numbers $p$ and $q$, the only real numbers
$t$ for which both $p^{t}$ and $q^{t}$ are rational are the positive
integers.\end{assumption}
\begin{conclusion}
$q_{i}$ is prime for all $i$, rf. \cite[Colossally abundant number]{Wikipedia}.

\newpage{}
\end{conclusion}

\section{Subsequent Maximisers\label{sec:Subsequent Maximisers}}

\subsection{Extending Robin's Method}

Robin's crucial argument was Proposition~\ref{prop:maximising-CA-numbers}.
A Transfer to Condition~\ref{condition:sufficient} follows.
\begin{defn}
Let $\mathcal{Q}_{i,k}:=\prod\limits _{j=1}^{k}q_{i+j}$, $\mathcal{R}_{i,k}:=\prod\limits _{j=1}^{k}R_{i+j-1}\left(n_{i+j},\varepsilon_{i+j}\right)$,\foreignlanguage{english}{
$\mathcal{L}_{i,k}:=\prod\limits _{j=1}^{k}L_{i+j-1}\left(n_{i+j}\right)=\frac{\ln\ln n_{i+k}}{\ln\ln n_{i}}$,}
and\foreignlanguage{english}{ $\mathcal{G}_{i,k}:=\prod\limits _{j=1}^{k}g_{i+j}=\frac{\sigma_{-1}\left(n_{i+k}\right)}{\sigma_{-1}\left(n_{i}\right)}$}.
Contextually write $\mathcal{X}_{i,\cdot}$ for $\left(\mathcal{X}_{i,j}\right)_{j\in\mathbb{N}}$
or $\left\{ \mathcal{X}_{i,j}\right\} _{j\in\mathbb{N}}$ if $\mathcal{X}\in\left\{ \mathcal{R},\mathcal{G},\mathcal{L},\mathcal{Q},\mathcal{D}\right\} $
where $\mathcal{D}_{i,k}:=\mathcal{G}_{i,k}-\mathcal{L}_{i,k}$. Put
$k_{i}:=\inf\left\{ k\in\mathbb{N};\mathcal{R}_{i,k}\geq1\right\} $.\end{defn}
\begin{fact*}
$\frac{\mathcal{G}_{i,k}}{\mathcal{L}_{i,k}}=\frac{\sigma_{-1}\left(n_{i+k}\right)}{\ln\ln n_{i+k}}\cdot\frac{\ln\ln n_{i}}{\sigma_{-1}\left(n_{i}\right)}=\frac{X\left(n_{i+k}\right)}{X\left(n_{i}\right)}$.\end{fact*}
\begin{lem}
\label{lem:Robins-product-separated}$\mathcal{R}_{i,k}\cdot\mathcal{L}_{i,k}=\mathcal{G}_{i,k}$
and $X\left(n_{i}\right)\cdot\mathcal{R}_{i,k}=X\left(n_{i+k}\right)$.\end{lem}
\begin{proof}
By induction on $k$ using Lemma~\ref{note:convexity}, \#4.\end{proof}
\begin{cor}
\label{cor:maximum-number-of-multipliers}If $k_{i}<\infty$ then
$\mathcal{G}_{i,k_{i}}\geq\mathcal{L}_{i,k_{i}}$ so $\mathcal{G}_{i,k}<\mathcal{L}_{i,k}$
for all $k$ if $k_{i}=\infty$ but $k\geq k_{i}$ if $\mathcal{D}_{i,k}\geq0$.\end{cor}
\begin{thm}
\label{thm:Exceptionality-by-Number-of-Multipliers}$n_{i}$ is exceptional
if and only if $k_{i}=\infty$ and $i>8$.\end{thm}
\begin{proof}
$X\left(n_{i}\right)>X\left(n_{i+k}\right)$ for all $k$ iff $n_{i}$
is exceptional. On the the other hand $\mathcal{R}_{i,k}<1$ for all
$k$ iff $k_{i}=\infty$ and the claim follows with Lemma~\ref{lem:Robins-product-separated}.\end{proof}
\begin{condition}
\label{condition:Robins-Multipliers}$k_{i}<\infty$ for all $i>8$.\end{condition}
\begin{thm}
\label{thm:extended-Robins-Method}Conditions~\ref{condition:Robins-Multipliers}
and~\ref{condition:separated} are equivalent to Condition~\ref{condition:sufficient}
with $n=n_{i}$ and $x=\mathcal{Q}_{i,k_{i}}$.\end{thm}
\begin{proof}
For the bounds on $i$ consider $n_{8}=5040$ and section~\ref{sub:Number-Crunching}.

All other statements follow from Lemma~\ref{lem:Robins-product-separated}.\end{proof}
\begin{condition}
\label{condition:separated}For every $n_{i}$ with $i>143215$ there
is some $k$ such that $\mathcal{D}_{i,k}\geq0$.
\end{condition}

\begin{figure}
\includegraphics[width=7cm]{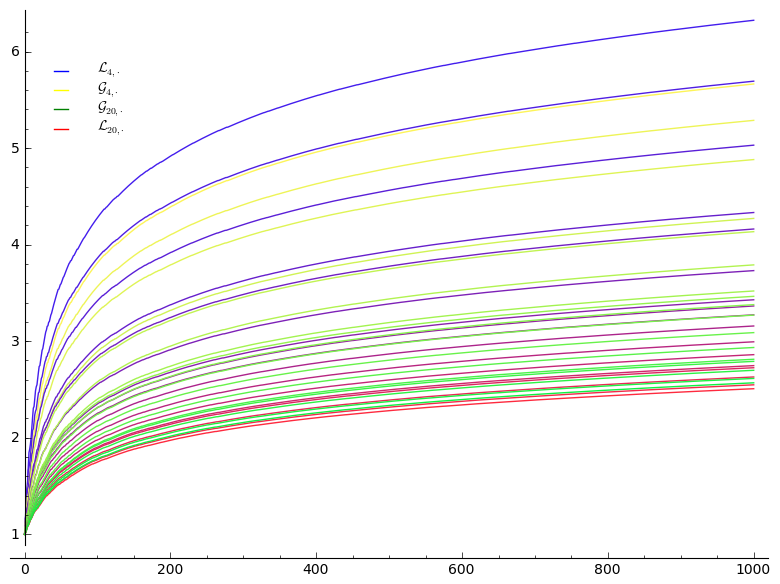}\includegraphics[width=7cm]{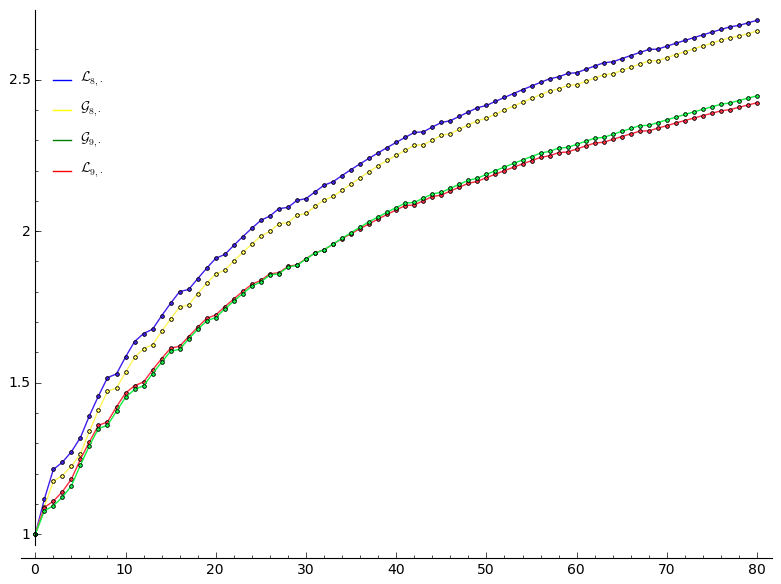}

\caption{``The $\mathcal{G}_{i,\cdot}$'s are enclosed by the $\mathcal{L}_{i,\cdot}$'s''}
\label{figure:Exceptionality-dominates-RDS}
\end{figure}

\subsection{\label{sub:Number-Crunching}Number Crunching}
\begin{itemize}
\item [1.] Sage led me to my first results, \cite{Sage:2011}. My algorithm
passed the CA numbers in table~\ref{tab:Statistics-CA-numbers}.
According to \cite{Noe:2005,Noe:2009} T.D.Noe's form $\left(\alpha_{v}\right)_{v=1}^{n}$
represents $\prod\limits _{v=1}^{n}\prod\limits _{j=\pi(\alpha_{v+1})+1}^{\pi(\alpha_{v})}p_{j}^{v}$
if $\alpha$ contains no zeros and $\alpha_{n+1}=0\wedge p_{0}=1$.
Thus $\qquad n_{8}=\left(7,3,0,2\right)$, $n_{508}=\left(3257,73,19,7,5,0^{2},3,0^{5},2\right)$,
\\
$n_{9}=(11,3,0,2)$, $n_{42}=(101,13,5,0,3,0^{2},2)$, $n_{2386}=\left(20359,193,37,13,7,0,5,0^{2},3,0^{6},2\right)$,\\
$n_{143215}=(1911373,1907,173,47,23,13,0,7,0,5,0^{4},3,0^{8},2)$,
and $n_{13}=\left(13,5,3,0,2\right)$.
\end{itemize}
\begin{table}[H]
\begin{centering}
\begin{tabular}{|c|c|c|c|c|c|c|c|}
\hline 
 & $n_{8}$ & $n_{508}$ & $n_{9}$ & $n_{13}$  & $n_{42}$ & $n_{2386}$ & $n_{143215}$\tabularnewline
\hline 
\hline 
$v_{2}(\cdot)\quad$ & 4 & 14 & 4 & 5 & 8 & 17 & 24\tabularnewline
\hline 
$\ln\quad$ & 8.5251 & 3274.0 & 10.9230 & 16.889 & 107.7176 & 20432.8 & 1912150.6\tabularnewline
\hline 
$\ln(P_{1}\left(\cdot\right))\quad$ & 1.9459 & 8.0885 & 2.3978 & 2.5649 & 4.6151 & 9.9212 & 14.4633\tabularnewline
\hline 
$\ln\left(-\frac{1}{\varepsilon\ln\left(\varepsilon\right)}\right)\quad$ & 1.9356 & 7.8588 & 2.1174 & 2.5342 & 4.3330 & 9.7132 & 14.2938\tabularnewline
\hline 
$\ln\ln\quad$ & 2.1430 & 8.09\textbf{3}7 & 2.3908 & 2.8266 & 4.6795 & 9.9249 & 14.4637\tabularnewline
\hline 
$k_{i}$ & $\infty$ & 1 & 33 & 1 & 1 & 1 & 1\tabularnewline
\hline 
$\lg\quad$ & 3.7024 & 1421.9 & 4.7438 & 7.335 & 46.7811 & 8873.60 & 830436.46\tabularnewline
\hline 
$\lg\lg\quad$ & 0.5684 & 3.1528 & 0.6761 & 0.8654 & 1.6700 & 3.9480 & 5.9193\tabularnewline
\hline 
$\sigma_{-1}\quad$ & 3.8380 & 14.3887 & 4.1870 & 4.8559 & 8.1962 & 17.6663 & 25.7599\tabularnewline
\hline 
$X\quad$ & 1.79097 & 1.7777 & 1.751\textbf{2} & 1.7179 & 1.751\textbf{5} & 1.7\textbf{8}00001 & 1.78\textbf{1}0000003\tabularnewline
\hline 
$B\quad$ & 1.9047 & 1.79096 & 1.8944 & 1.8621 & 1.8106 & 1.7877 & 1.7842\tabularnewline
\hline 
$\ln\ln B^{-1}\circ X\quad$ & \selectlanguage{english}%
8.09\textbf{1}3\selectlanguage{american}%
 & \multicolumn{6}{c}{undefined since $X\left(n\right)<e^{\gamma}=1.7811$}\tabularnewline
\hline 
\end{tabular}
\par\end{centering}

\caption{\label{tab:Statistics-CA-numbers}Statistics of some CA numbers}
\end{table}

\begin{thm}
For every CA $n_{i}$ if $i\le143215$ there is a subsequent CA $n_{i+j}$
such that $X(n_{i})<X(n_{i+j})$.
\begin{proof}
 $\textrm{RIE}\left(n\right)$ was confirmed in every loop. (Not shown
in Appendix~\ref{sec:Implementation}.)
\end{proof}
\end{thm}
\begin{itemize}
\item [2.] Keith Briggs reported to me: \textquotedbl{}E.g. the following
is a CA number:

\[
\begin{array}{l}
n=2^{41}\cdot3^{25}\cdot5^{17}\cdot7^{14}\cdot11^{11}\cdot13^{10}\cdot17^{9}\cdots23^{8}\cdots37^{7}\cdots53^{6}\cdots101^{5}\cdots\\
\hphantom{n_{\textrm{Br}}=\cdots101^{5}}\cdots239^{4}\cdots887^{3}\cdots7789^{2}\cdots562399^{1}\cdots162216342187^{0}
\end{array}
\]

with loglog $n$ about 26.\textquotedbl{} Denote it by $n_{\textrm{Br}}$
and call the tuple \foreignlanguage{english}{(2, $0^{15}$, 3, $0^{7}$,
5, $0^{2}$, 7, $0^{2}$, 11, 13, 17, 23, 37, 53, 101, 239, 887, 7789,
562399, 162216342187)} \emph{bottom-up form}. Since he reached about
$10^{10^{10}}$ according to \cite[§3]{Briggs:2006} 
\[
\lg\lg n_{\textrm{Br}}\approx\lg\left(\lg\left(\exp\left(\exp\left(26\right)\right)\right)\right)=\lg\left(\frac{\exp\left(26\right)}{\ln10}\right)=\frac{1}{\ln10}\left(26-\ln\ln10\right)\approx10.9294
\]
reveals $\ln\ln n_{\textrm{Br}}\approx26$ in compliance with \cite[Section 10.1.2]{NIST:2008}
and Fact~\ref{fact:basic-CA-number-properties}. This was significantly
more than the theoretically obtained bound $10^{8576}<e^{19747}$
from \cite[\mbox{Corollary 1}]{Caveney:2011} which my rather short
calculation capped, too. Noe's \emph{top-down} representation is 
\[
\left(162216342179,562361,7759,883,233,97,47,31,19,13,11,0^{2},7,0^{2},5,0^{7},3,0^{15},2\right).
\]

\begin{thm}
For every CA $n_{i}$ such that $\ln\ln n_{i}\leq25<26$ there is
a subsequent CA $n_{i+j}$ such that $X(n_{i})<X(n_{i+j})$.
\end{thm}
\end{itemize}
\newpage{}

\begin{table}[H]
\begin{longtable}{|c|c|c|c|c|c|c|c|c|c|c|c|}
\hline 
$i+j$ & $v_{2}$ & $q_{i+j}$ & $p$ & $\ln$ & $v_{i,j}$ & \selectlanguage{english}%
$g_{i+j}$\selectlanguage{american}%
 & $\mathcal{G}_{8,j}$ & $\varepsilon_{i+j}$ & $ll$ & $\mathcal{L}_{8,j}$ & $-\mathcal{D}_{9,j}$\tabularnewline
\hline 
\endhead
\hline 
8 & 4 & 2 & 7 & 8.5 & 4 & 31:30 & 1 & 4.73$d^{2}$ & 2.143 & 1 & \tabularnewline
\hline 
9 & 4 & 11 & 11 & 10.9 & 1 & 12:11 & 1.090 & 3.62$d^{2}$ & 2.390 & 1.115 & 1\tabularnewline
\hline 
10 & 4 & 13 & 13 & 13.4 & 1 & 14:13 & 1.174 & 2.88$d^{2}$ & 2.601 & 1.214 & 1.13$d^{2}$\tabularnewline
\hline 
11 & 5 & 2 & 13 & 14.1 & 5 & 63:62 & 1.193 & 2.31$d^{2}$ & 2.651 & 1.237 & 1.48$d^{2}$\tabularnewline
\hline 
12 & 5 & 3 & 13 & 15.2 & 3 & 40:39 & 1.224 & 2.30$d^{2}$ & 2.726 & 1.272 & 1.80$d^{2}$\tabularnewline
\hline 
13 & 5 & 5 & 13 & 16.8 & 2 & 31:30 & 1.265 & 2.03$d^{2}$ & 2.826 & 1.319 & 2.25$d^{2}$\tabularnewline
\hline 
14 & 5 & 17 & 17 & 19.7 & 1 & 18:17 & 1.339 & 2.01$d^{2}$ & 2.981 & 1.391 & 1.91$d^{2}$\tabularnewline
\hline 
15 & 5 & 19 & 19 & 22.6 & 1 & 20:19 & 1.410 & 1.74$d^{2}$ & 3.120 & 1.456 & 1.27$d^{2}$\tabularnewline
\hline 
16 & 5 & 23 & 23 & 25.8 & 1 & 24:23 & 1.471 & 1.35$d^{2}$ & 3.250 & 1.516 & 1.07$d^{2}$\tabularnewline
\hline 
17 & 6 & 2 & 23 & 26.4 & 6 & 127:126 & 1.483 & 1.12$d^{2}$ & 3.276 & 1.529 & 1.11$d^{2}$\tabularnewline
\hline 
18 & 6 & 29 & 29 & 29.8 & 1 & 30:29 & 1.534 & 1.00$d^{2}$ & 3.396 & 1.584 & 1.42$d^{2}$\tabularnewline
\hline 
19 & 6 & 31 & 31 & 33.2 & 1 & 32:31 & 1.583 & 9.24$d^{3}$ & 3.505 & 1.635 & 1.44$d^{2}$\tabularnewline
\hline 
20 & 6 & 7 & 31 & 35.2 & 2 & 57:56 & 1.612 & 9.09$d^{3}$ & 3.562 & 1.662 & 1.22$d^{2}$\tabularnewline
\hline 
21 & 6 & 3 & 31 & 36.3 & 4 & 121:120 & 1.625 & 7.55$d^{3}$ & 3.592 & 1.676 & 1.27$d^{2}$\tabularnewline
\hline 
22 & 6 & 37 & 37 & 39.9 & 1 & 38:37 & 1.669 & 7.38$d^{3}$ & 3.687 & 1.720 & 1.21$d^{2}$\tabularnewline
\hline 
23 & 6 & 41 & 41 & 43.6 & 1 & 42:41 & 1.710 & 6.48$d^{3}$ & 3.776 & 1.762 & 1.19$d^{2}$\tabularnewline
\hline 
24 & 6 & 43 & 43 & 47.4 & 1 & 44:43 & 1.749 & 6.11$d^{3}$ & 3.859 & 1.800 & 1.00$d^{2}$\tabularnewline
\hline 
25 & 7 & 2 & 43 & 48.1 & 7 & 255:254 & 1.756 & 5.66$d^{3}$ & 3.873 & 1.807 & 9.82$d^{3}$\tabularnewline
\hline 
26 & 7 & 47 & 47 & 51.9 & 1 & 48:47 & 1.794 & 5.46$d^{3}$ & 3.950 & 1.843 & 7.75$d^{3}$\tabularnewline
\hline 
27 & 7 & 53 & 53 & 55.9 & 1 & 54:53 & 1.828 & 4.70$d^{3}$ & 4.024 & 1.877 & 7.51$d^{3}$\tabularnewline
\hline 
28 & 7 & 59 & 59 & 60.0 & 1 & 60:59 & 1.858 & 4.12$d^{3}$ & 4.094 & 1.910 & 8.54$d^{3}$\tabularnewline
\hline 
29 & 7 & 5 & 59 & 61.6 & 3 & 156:155 & 1.870 & 3.99$d^{3}$ & 4.121 & 1.923 & 8.61$d^{3}$\tabularnewline
\hline 
30 & 7 & 61 & 61 & 65.7 & 1 & 62:61 & 1.901 & 3.95$d^{3}$ & 4.185 & 1.953 & 7.51$d^{3}$\tabularnewline
\hline 
31 & 7 & 67 & 67 & 69.9 & 1 & 68:67 & 1.930 & 3.52$d^{3}$ & 4.247 & 1.982 & 7.42$d^{3}$\tabularnewline
\hline 
32 & 7 & 71 & 71 & 74.2 & 1 & 72:71 & 1.957 & 3.28$d^{3}$ & 4.306 & 2.009 & 7.25$d^{3}$\tabularnewline
\hline 
33 & 7 & 73 & 73 & 78.4 & 1 & 74:73 & 1.984 & 3.17$d^{3}$ & 4.363 & 2.035 & 6.18$d^{3}$\tabularnewline
\hline 
34 & 7 & 11 & 73 & 80.8 & 2 & 133:132 & 1.999 & 3.14$d^{3}$ & 4.393 & 2.049 & 4.99$d^{3}$\tabularnewline
\hline 
35 & 7 & 79 & 79 & 85.2 & 1 & 80:79 & 2.024 & 2.87$d^{3}$ & 4.445 & 2.074 & 3.79$d^{3}$\tabularnewline
\hline 
36 & 8 & 2 & 79 & 85.9 & 8 & 511:510 & 2.028 & 2.82$d^{3}$ & 4.453 & 2.078 & 3.54$d^{3}$\tabularnewline
\hline 
37 & 8 & 83 & 83 & 90.3 & 1 & 84:83 & 2.052 & 2.71$d^{3}$ & 4.503 & 2.101 & 2.11$d^{3}$\tabularnewline
\hline 
38 & 8 & 3 & 83 & 91.4 & 5 & 364:363 & 2.058 & 2.50$d^{3}$ & 4.516 & 2.107 & 1.98$d^{3}$\tabularnewline
\hline 
39 & 8 & 89 & 89 & 95.9 & 1 & 90:89 & 2.081 & 2.48$d^{3}$ & 4.563 & 2.129 & 8.18$d^{4}$\tabularnewline
\hline 
40 & 8 & 97 & 97 & 100 & 1 & 98:97 & 2.103 & 2.24$d^{3}$ & 4.610 & 2.151 & 6.26$d^{4}$\tabularnewline
\hline 
41 & 8 & 13 & 97 & 103 & 2 & 183:182 & 2.114 & 2.13$d^{3}$ & 4.635 & 2.163 & 5.70$d^{4}$\tabularnewline
\hline 
42 & 8 & 101 & 101 & 107 & 1 & 102:101 & 2.135 & 2.13$d^{3}$ & 4.679 & 2.183 & \textbf{-3.05}$d^{4}$\tabularnewline
\hline 
$\vdots$ &  &  &  &  &  &  & $\vdots$ &  &  & $\vdots$ & \tabularnewline
\hline 
507 & 14 & 3253 & 3253 & 3265 & 1 & 3254:3253 & 3.747 & 3.80$d^{5}$ & 8,091 & 3.775 & -5.12$d^{2}$\tabularnewline
\hline 
508 & 14 & 3257 & 3257 & 3274 & 1 & 3258:3257 & 3.748 & 3.79$d^{5}$ & \textbf{8.093} & 3.776 & -5.12$d^{2}$\tabularnewline
\hline 
\end{longtable}

\caption{\label{tab:GL-sequence}Quotients $q_{i+j}$ of CA numbers, note $k_{9}=33$,
\cite[A073751]{OEIS} and abbreviate $p=P_{1}\left(n_{i+j}\right)$,
$ll=\ln\ln n_{i+j}$, $v_{i,j}=v_{q_{i+j}}\left(n_{i+j}\right)$ and
$d=10^{-1}$.}
\end{table}

\newpage{}

\section{The Question of Life\label{sec:The-Question-of-Life}}

The next Lemma has a long track in my notes since the preprint of
\cite[Lemma 6.1]{Choie:2007} was not hard to complement in the present
setup.
\begin{prop}
$k_{i}=1$ if
\[
\left(\sum\limits _{l=1}^{v_{i,1}}q_{i+1}^{l}\right)\ln q_{i+1}\leq\ln n_{i}\cdot\ln\ln n_{i}\,.
\]
\end{prop}
\begin{proof}
$\mathcal{G}_{i,1}=1+\left(\sum\limits _{l=1}^{v_{i,1}}q_{i+1}^{l}\right)^{-1}\geq1+\frac{\ln\mathcal{Q}_{i,1}}{\ln n_{i}\cdot\ln\ln n_{i}}$
by assumption. A Taylor approximation of $\log_{\ln n_{i}}\left(\ln n_{i}+h\right)$
for $h=\ln\left(\mathcal{Q}_{i,k}\right)$ has remainder term $-\frac{1}{2}\cdot\left(\frac{\ln\mathcal{Q}_{i,k}}{\ln n_{i}+\vartheta\ln\mathcal{Q}_{i,k}}\right)^{2}\cdot\frac{1}{\ln\ln n_{i}}<0$.
Therefore the case $k=1$ yields $\mathcal{D}_{i,1}\geq0$.
\end{proof}
Most recently, Morkotun demonstrated in \cite[Theorem 2]{Morkotun:2013}
how to include all prime factors of $n_{i}$ without requiring the
CA or SA property of $n_{i}$. The existence of a sequence on which
$X$ increases follows from Grönwall's theorem if exceptional numbers
do not exist. But if there are exceptional numbers $X$ will stop
to take larger values because of Robin's unconditional bound. \cite[(4)]{Morkotun:2013}
is often met but once in a while abundant numbers have prime factors
larger than $\ln n$ in which case it would have been possible to
argue with Lemma~\ref{lem:Lemma_from_Choie2007} below. Likewise
it is very possible that the sequence of $i$'s with $k_{i}>1$ is
infinite although the gaps between regions with $k_{i}>1$ may be
large.
\begin{lem}
\label{lem:Lemma_from_Choie2007}\cite[Lemma 6.1]{Choie:2007}\textbf{:}
If $\ln n<P_{1}\left(n\right)$ for a $t$-free $n$ with $t\geq2$
then $RIE\left(n\right)$.
\end{lem}
However, applying either the Proposition above or Morkotun's condition
of RIE it is sufficient to consider the greatest primes $p$ with
$v_{p}\left(n\right)=v$ for each valuation $v$ between $1=v_{P_{1}\left(n_{i}\right)}\left(n_{i}\right)$
and $v_{2}\left(n_{i}\right)$. This is reflected by Noe's representation
of SA numbers in section~\ref{sub:Number-Crunching}. In each loop
the CA numbers algorithm chooses the $q_{i}$ for which $\varepsilon_{i+1}=\frac{1}{\ln q_{i+1}}\ln g_{i+1}$
is maximal when $q_{i+1}$ varies over the primes $p$ for which $n=n_{i}p$
meets $v_{q}\left(n\right)<v_{p}\left(n\right)$ if $q>p$.

\subsection{Extremely Abundant Numbers\label{sec:Extremely-Abundant-Numbers}}

The recent papers \cite{Nazardonyavi:Yakubovich:2013,Nazardonyavi:Yakubovich:2013:Delicacy}
will be summarised in the context of the present one. After some quotations
from the follow-up paper this section employs the numbers of the text
modules in \cite{Nazardonyavi:Yakubovich:2013}. 
\begin{defn*}
\textbf{(2.1):} \cite[Def. 1.2, 1.3, and 1.8]{Nazardonyavi:Yakubovich:2013:Delicacy}
\begin{enumerate}
\item $n\in XA$ iff $X\left(n\right)>X\left(m\right)$ for $m\in\left[10080,n\right]_{\mathbb{N}}$,
\item $n\in XA'$ iff $\frac{\sigma_{-1}(n)}{\sigma_{-1}(m)}>1+\frac{\ln n-\ln m}{\ln n\cdot\ln\ln m}$
for $m=\max\left\{ k\in XA';k<n\right\} $, and
\item $n\in XA''$ iff $\frac{\sigma_{-1}(n)}{\sigma_{-1}(m)}>1+2\frac{\ln n-\ln m}{\left(\ln n+\ln m\right)\cdot\ln\ln m}$
for $m=\max\left\{ k\in XA'';k<n\right\} $.
\end{enumerate}
\end{defn*}
\begin{conclusion*}
\cite[(6)]{Nazardonyavi:Yakubovich:2013:Delicacy} $XA\subseteq XA'\subseteq SA$.\end{conclusion*}
\begin{lem*}
\cite[Lemma 1.4]{Nazardonyavi:Yakubovich:2013:Delicacy} $XA'$ is
well-defined.\end{lem*}
\begin{thm*}
\cite[Theorem 1.7]{Nazardonyavi:Yakubovich:2013:Delicacy} $\left|XA'\right|=\infty$.
\end{thm*}

\begin{thm*}
\textbf{(2.3):} The least $n>5040$ such that $RIE\left(n\right)$
is false is in XA.
\end{thm*}

\begin{thm*}
\textbf{(2.4):} $RH\Longleftrightarrow\left|XA\right|=\infty$.
\end{thm*}

\begin{thm*}
\textbf{(4.28):} $RH\Longrightarrow\left|CA\cap XA\right|=\infty$.
\end{thm*}

\begin{thm*}
\textbf{(4.31):} $\left|n\in CA;\ln n<P_{1}\left(n\right)\right|=\infty$.
\end{thm*}

\begin{thm*}
\textbf{(4.32):} $n\in XA\Longrightarrow P_{1}\left(n\right)<\ln n$.
\end{thm*}

\begin{thm*}
\textbf{(4.34):} $\left|CA\setminus XA\right|=\infty$.
\end{thm*}
Essentially Theorem (2.4) asserts the necessity and sufficiency of
Condition~\ref{condition:sufficient} for RH. Theorem (4.28) provides
a necessary condition by restriction to CA numbers without mentioning
the obvious reverse implication in virtue of Theorem (2.4). In particular,
$\left|CA\setminus XA\right|=\infty$ and $\left|CA\cap XA\right|=\infty$
in case of RH make this case delicate. The advantage is the minimality
condition of Theorem (2.3) at the cost of loosing the availability
of an algorithm that computes the sequence of hypothetic counterexamples
of RIE.

\subsection{Stronger Ingredients\label{sub:Stronger-Ingredients}}

The goal of this section is to show that Condition~\ref{condition:separated}
is true. The subsection's title insinuates Assumption~\ref{Assumption:At-most-one-sign-change}.
The easiest step towards it was quoting Lemma~\ref{lem:Lemma_from_Choie2007}.

Thus RIE is not violated unless the prime divisors of $n_{i}$ cumulate
too densely and $n_{9}$ is the only CA number in section~\ref{sub:Number-Crunching}
with $P_{1}(n)>\ln n$. On the other hand by \cite[Thm 2]{Alaoglu:Erdoes:1944}
there must not be too many small prime divisors for RIE.
\begin{thm}
\label{thm:Mertens}(Mertens, \cite{Mertens:1874}, \cite[Thm 427-429]{Hardy:Wright:1979})
The \emph{Meissel–Mertens constant} is given by
\[
B_{1}=\lim_{n\to\infty}\left(\sum_{p\le n}\frac{1}{p}-\ln\ln n\right)=\gamma+\sum_{p}\left(\ln\left(1-\frac{1}{p}\right)+\frac{1}{p}\right)=0.26149\,72128\,47642....
\]
\end{thm}
\begin{note*}
Interesting additional references are \cite{Villarino:2005} and \cite[§1 (1)]{Lindqvist:1997:Remainder}.
\end{note*}
It can be considered reasonable to assume that $\mathcal{G}_{i,\cdot}$
grows at least as fast as $\mathcal{L}_{i,\cdot}$ for increasing
$k$. This conjecture is based on Theorem~\ref{thm:Mertens} and
the culmination of the work on the asymptotics of $p_{k}$, \cite{Rosser:1939:nthprime,Rosser:Schoenfeld:1975,Robin:1983:estim:theta,Massias:1996:Bornes}
in P. Dusart's statement $p_{k}\geq k\left(\ln k+\ln\ln k-1\right)$,
\cite{Dusart:1999} after $p_{n}\geq n\left(\ln n+\ln\ln n-1+o\left(\frac{\ln\ln n}{\ln n}\right)\right)+k$
had become available in \cite{Ribenboim:1991:little} without guaranteeing
Dusart's lower bound, yet.
\begin{lem}
\label{lem:GL-sequences-approach-infinity}$\mathcal{G}_{i,k}\rightarrow\infty$
and $\mathcal{L}_{i,k}\rightarrow\infty$ as $k\rightarrow\infty$
and $\mathcal{D}_{i,1}>0$ if and only if\\
 $g_{i+1}-1>2\mathrm{artanh}\left(\frac{\ln q_{i+1}}{2\ln n_{i}+\ln q_{i+1}}\right)$.\end{lem}
\begin{proof}
For every prime $p>P_{1}\left(n_{i}\right)$ there is some $k$ such
that $v_{p}\left(n_{i+k}\right)=1$. Therefore $\frac{1}{p}$ occurs
as a summand when expanding \foreignlanguage{english}{$\mathcal{G}_{i,k}:=\prod\limits _{j=1}^{k}\left(1+\frac{1}{g_{v_{i,j}}\left(x\right)}\right)$
which must be a bound }for the summed reciprocals of subsequent primes,
i.e. $R_{i,k}:=\sum\left\{ p^{-1};P_{1}\left(n_{i}\right)<p<P_{1}\left(n_{i+k}\right)\right\} <\mathcal{G}_{i,k}$.
But $R_{i,k}\rightarrow\infty$ as $k\rightarrow\infty$ by Theorem~\ref{thm:Mertens}.
With $y=\frac{1+x}{1-x}$ iff $x=\frac{y-1}{y+1}$ and $\ln(y)=2\operatorname{artanh}\left(x\right)$

\[
\mathcal{L}_{i,k}=\log_{\ln n_{i}}\left(\ln n_{i+k}\right)=1+\log_{\ln n_{i}}\left(1+\frac{\ln\mathcal{Q}_{i,k}}{\ln n_{i}}\right)
\]
can be expanded to
\[
\log_{a}\left(a+z\right)-1=\frac{2}{\ln a}\left(\frac{z}{2a+z}+\frac{1}{3}\left(\frac{z}{2a+z}\right)^{3}+\frac{1}{5}\left(\frac{z}{2a+z}\right)^{5}+\cdots\right)
\]
for $a=\ln n_{i}>0$ and $z=\ln\mathcal{Q}_{i,k}\geq-a$, rf. \foreignlanguage{english}{\cite[4.1.29]{Abramowitz:Stegun:1964}}.
The conclusion can be drawn by using $\mathcal{G}_{i,1}-1=g_{i+1}-1$.
(I came across the last formula in some book dealing with elliptic
functions, too. Unfortunately I seem to be unable to find it again.)\end{proof}
\begin{assumption}
\label{Assumption:At-most-one-sign-change}$\mathcal{D}_{i,\cdot}$
has at most one change of sign.
\end{assumption}
This is quite a strong assumtion. Given Littlewood's theorem on the
difference $\pi(x)-\operatorname{Li}(x)$, \cite{Littlewood:1914}
and Robin's theorem on $\sum\left\{ p^{-1};n\geq p\in\mathbb{P}\right\} -\ln\ln n-B_{1}$,
\cite[Théorème 2]{Robin:1983:ordre:maximum} it makes sense to assume
the opposite. Viewing $B_{1}$ as safty buffer between $\sum\left\{ p^{-1};n\geq p\in\mathbb{P}\right\} $
and $\ln\ln n$ may render Assumption~\ref{Assumption:At-most-one-sign-change}
reasonable and it could probably be deduced from \cite[Thm 4.21]{Nazardonyavi:Yakubovich:2013}
if its implied constant is not too large.

In virtue of Theorem~\ref{thm:extended-Robins-Method} it is easy
to derive Condition~\ref{condition:sufficient} from Assumption~\ref{Assumption:At-most-one-sign-change}.
For this the key is to derive 
\[
\begin{array}{rl}
\forall k\geq k_{i}:\;\ln q_{i+k+1} & >\left(\ln n_{i+1}\right)^{\mathcal{G}_{i+1,k}}-\left(\ln n_{i}\right)^{\mathcal{G}_{i,k}}\\
 & >\left(\ln q_{i+1}\right)^{\mathcal{G}_{i+1,k}}+\left(\ln n_{i}\right)^{\mathcal{G}_{i+1,k}}-\left(\ln n_{i}\right)^{\mathcal{G}_{i,k}}
\end{array}
\]
from $k_{i}<\infty$ and $k_{i+1}=\infty$. This in turn can be done
with the equivalence of $\mathcal{D}_{i,k}>0$ and $\ln\mathcal{Q}_{i,k}<\left(\ln n_{i}\right)^{G\left(q_{i+1},v_{i,1}\right)\cdot\cdots\cdot G\left(q_{i+k},v_{i,k}\right)}-\ln n_{i}$.
Therefore if $k_{i+1}=\infty$ then small primes could occur only
finitely often in the sequence $\left(q_{i}\right)_{i=1}^{\infty}$
which contradicts the CA numbers algorithm.\foreignlanguage{english}{
It should have been possible to reduce }Assumption~\ref{Assumption:At-most-one-sign-change}
to requiring that $\mathcal{D}_{i,\cdot}$ has only finitely many
changes of sign the last of which being from - to +. However, Assumption~\ref{Assumption:At-most-one-sign-change}
remains undecided.

The idea that $\mathcal{L}_{i,k}$ converges faster to 1 than $\mathcal{G}_{i,k}$
as $i\rightarrow\infty$ - or equivalently $\mathcal{G}_{i,\cdot}$
does not grow slower than $\mathcal{L}_{i,\cdot}$ - was motivated
by Figure~\ref{figure:Exceptionality-dominates-RDS} which covers
a much too small part for a reasonable confidence level. Another way
to express the higher speed of convergence is the next claim which
is equivalent to Claim~\ref{claim:general-validity-of-RIE}.
\begin{claim}
\label{claim:induction-from-oscillation}$i>8$ and $k_{i+1}<\infty$
if $k_{i}<\infty$ for some $i>2$.
\end{claim}
By Theorem~\ref{thm:extended-Robins-Method} this claim is sufficient
for Claim~\ref{claim:general-validity-of-RIE}. Conversely it is
necessary as can be seen with Condition~\ref{condition:Robins-Multipliers}
and section~\ref{sub:Number-Crunching}, too.

\subsection{Oscillation Theorems\label{sub:Oscillation-Theorems}}

Clearly, everything works fine if RH is true. An indirect proof with
osciallation theorems like \cite[Corollaire 1]{Nicolas:1983}, \cite[Sec. 4]{Robin:1983:ordre:maximum},
or \cite{Grosswald:1972,Anderson:Stark:1981} will be proposed by
showing that the minimal oscillations force $\mathcal{G}_{i,k}$ above
$\mathcal{L}_{i,k}$ for $k$ sufficiently large. After all, Voros
reported in \cite[Ch. 11]{Voros:2010} that the amplitude of Keiper's
sequence $\left(\lambda_{n}\right)_{n}$ grows exponentially, \cite{Keiper:1992,LiXianJin:1997,Bombieri:Lagarias:1999}.
\begin{defn}
\cite{Hardy:Littlewood:1914,Landau:1924}
\begin{enumerate}
\item $f(x)=\Omega_{\pm}(g(x))$ $\left(x\to\infty\right)$ means that both
$f(x)=\Omega_{+}(g(x))$ and $f(x)=\Omega_{-}(g(x))$$\left(x\to\infty\right)$
are valid where

\begin{itemize}
\item $f(x)=\Omega_{+}(g(x))\;\left(x\to\infty\right)\quad\rightleftharpoons\quad\limsup\limits _{x\to\infty}\frac{f\left(x\right)}{g\left(x\right)}>0\;\textrm{and}$
\item $f\left(x\right)=\Omega_{-}(g(x))\;\left(x\to\infty\right)\quad\rightleftharpoons\quad\liminf\limits _{x\to\infty}\frac{f\left(x\right)}{g\left(x\right)}<0\,.$
\end{itemize}
\item For $\theta:=\sup\left\{ \Re\left(z\right);z\in\mathbb{C},\,\zeta\left(z\right)=0\right\} $
the set $\left[1-\theta,\theta\right]+i\mathbb{R}$ will be called
\emph{very critical strip}.
\end{enumerate}
\end{defn}
Obviously RH is true if and only if the very critical strip coincides
with the critical line.
\begin{note}
\ 
\begin{enumerate}
\item By definition $f(x)=\Omega_{+}(g(x))$ and $f(x)=\Omega_{-}(g(x))$
holf if and only if $f\left(x\right)<o\left(g\left(x\right)\right)$
and $f\left(x\right)>o\left(g\left(x\right)\right)$ are respectively
false whereas $f(x)=\Omega(g(x))$ means that $f(x)\neq o(g(x))$
is false, i.e. on some sequence of $x$'s $f$ is at least of order
$g$.
\item D. E. Knuth preferred $f(x)=\Omega(g(x))\rightleftharpoons g(x)=O(f(x))$
which is not quite the same, \cite{Knuth:1976}.
\end{enumerate}
\end{note}
\begin{assumption}
\label{assumption:the-least-exceptional-number}

For the rest of the section let $b<\frac{1}{2}$ be in the very critical
strip and $n_{i}$ the least exceptional number which fixes an index
$i$ for the remaining section.\end{assumption}
\begin{thm}
\label{thm:Robin-oscillation-theorem}\cite[§4, Proposition]{Robin:1984:grandes:valeurs}

Under Assumption~\ref{assumption:the-least-exceptional-number} the
following holds true for CA numbers $n$.

\[
X\left(n\right)=e^{\gamma}\cdot\left(1+\Omega_{\pm}\left(\left(\ln n\right)^{-b}\right)\right)
\]
\end{thm}
\begin{cor}
\label{cor:Robin-oscillation-theorem} $\frac{X\left(n\right)}{e^{\gamma}}-1<o\left(\left(\ln n\right)^{-b}\right)$
and $\frac{X\left(n\right)}{e^{\gamma}}-1>o\left(\left(\ln n\right)^{-b}\right)$
are false, i.e. there are $\varepsilon_{1},\varepsilon_{2}>0$ such
that for every natural $N$ there are CA numbers $n_{1},n_{2}>N$
for which $\left(\frac{\sigma\left(n_{1}\right)}{e^{\gamma}\cdot n_{1}\cdot\ln\ln n_{1}}-1\right)\cdot\left(\ln n_{1}\right)^{b}>\varepsilon_{1}$
and $\left(\frac{\sigma_{-1}\left(n_{2}\right)}{\ln\ln n_{2}}-e^{\gamma}\right)\frac{\left(\ln n_{2}\right)^{b}}{e^{\gamma}}<-\varepsilon_{2}$
are true.\end{cor}
\begin{defn}
Let $\delta_{i}:=\frac{1}{2}\cdot\min\left(-\beta_{i},\alpha_{i}\right)>0$
for
\[
\begin{array}{rl}
\alpha_{i} & :=\limsup_{k\to\infty}\left(\left(\frac{\sigma_{-1}\left(n_{i+k}\right)}{\lambda_{i}\left(k\right)}-e^{\gamma}\right)\cdot\frac{\left(\ln n_{i+k}\right)^{b}}{e^{\gamma}}\right)>0\,,\\
\beta_{i} & :=\liminf_{k\to\infty}\left(\left(\frac{\sigma_{-1}\left(n_{i+k}\right)}{\lambda_{i}\left(k\right)}-e^{\gamma}\right)\cdot\frac{\left(\ln n_{i+k}\right)^{b}}{e^{\gamma}}\right)<0\,.
\end{array}
\]
The \emph{oscillation quotient} is $f_{i}\left(c,d\right):=\frac{\lambda_{i}\left(c\right)\left(1+\delta_{i}e^{-b\lambda_{i}\left(c\right)}\right)}{\lambda_{i}\left(d\right)\left(1-\delta_{i}e^{-b\lambda_{i}\left(d\right)}\right)}$
for \foreignlanguage{english}{$\lambda_{i}\left(k\right):=\ln\ln n_{i+k}$
}.\end{defn}
\begin{note}
Under Assumption~\ref{assumption:the-least-exceptional-number} infinitely
many changes of sign of $\left(\frac{\sigma_{-1}\left(n\right)}{\ln\ln n}-e^{\gamma}\right)\cdot e^{b\ln\ln n-\gamma}$
were established by Robin referring to the contributions of Nicolas,
Landau, and Grönwall.\end{note}
\begin{condition}
\label{condition:sufficient-by-oscillation}There are indices $c$
and $d$ with $c<d$ and $f_{i}\left(c,d\right)>1$.\end{condition}
\begin{thm}
\label{thm:sufficient-by-oscillation}There is no exceptional number
under Condition~\ref{condition:sufficient-by-oscillation}.\end{thm}
\begin{proof}
$\left(\frac{\sigma\left(n_{1}\right)}{e^{\gamma}\cdot n_{1}\cdot\ln\ln n_{1}}-1\right)\cdot\left(\ln n_{1}\right)^{b}>\varepsilon_{1}$
and $\left(\frac{\sigma_{-1}\left(n_{2}\right)}{\ln\ln n_{2}}-e^{\gamma}\right)\frac{\left(\ln n_{2}\right)^{b}}{e^{\gamma}}<-\varepsilon_{2}$\foreignlanguage{english}{
for some indices $c$ and $d$ follow from Corollary}~\ref{cor:Robin-oscillation-theorem}.
Therefore
\[
\begin{array}{rl}
\sigma_{-1}\left(n_{i+c}\right) & >e^{\gamma}\left(1+\delta_{i}e^{-b\lambda_{i}\left(c\right)}\right)\lambda_{i}\left(c\right)\quad\textrm{and}\\
\sigma_{-1}\left(n_{i+d}\right) & <e^{\gamma}\left(1+\delta_{i}e^{-b\lambda_{i}\left(d\right)}\right)\lambda_{i}\left(d\right)
\end{array}
\]
contradict $c<d$ as claimed by Condition~\ref{condition:sufficient-by-oscillation}
because a consequence after dividing the two inequalitites is
\[
1<f_{i}\left(c,d\right)\leq\frac{\sigma_{-1}\left(n_{i+c}\right)}{\sigma_{-1}\left(n_{i+d}\right)}=\begin{cases}
\left(\mathcal{G}_{i+c,d}\right)^{-1} & <1\quad;\, c<d\:,\\
\mathcal{G}_{i+d,c} & >1\quad;\, d<c\:.
\end{cases}
\]
\end{proof}
\selectlanguage{english}%
\begin{lem}
\textup{\label{lem:factorised-oscillation-quotient}The }\foreignlanguage{american}{\emph{oscillation
quotient}}\textup{ can be written as $f_{i}\left(c,d\right)=g\left(\lambda_{i}\left(c\right),\lambda_{i}\left(d\right)\right)$
for
\begin{equation}
g\left(\mu,\nu\right):=\frac{\mu\left(e^{b\cdot\left(\mu+\nu\right)}+\delta_{i}e^{b\cdot\nu}\right)}{\nu\left(e^{b\cdot\left(\mu+\nu\right)}-\delta_{i}e^{b\cdot\mu}\right)}=\frac{\mu}{\nu}\cdot\left(1+\delta_{i}e^{-b\mu}\right)\cdot\frac{e^{b\nu}}{e^{b\nu}-\delta_{i}}\,.\label{eq:factorisation-oscillation-quotient}
\end{equation}
}\end{lem}
\begin{proof}
Verify the factorisation by expanding the product on RHS. What remains
follows from 
\[
\begin{array}{rl}
f_{i}\left(c,d\right) & =\frac{\lambda_{i}\left(c\right)\left(1+\delta_{i}e^{-b\lambda_{i}\left(c\right)}\right)}{\lambda_{i}\left(d\right)\left(1-\delta_{i}e^{-b\lambda_{i}\left(d\right)}\right)}\\
 & =\frac{\lambda_{i}\left(c\right)\left(\left(e^{\lambda_{i}\left(c\right)}\cdot e^{\lambda_{i}\left(d\right)}\right)^{b}+\delta_{i}e^{b\lambda_{i}\left(d\right)}\right)}{\lambda_{i}\left(d\right)\left(\left(e^{\lambda_{i}\left(c\right)}\cdot e^{\lambda_{i}\left(d\right)}\right)^{b}-\delta_{i}e^{b\lambda_{i}\left(c\right)}\right)}=g\left(\lambda_{i}\left(c\right),\lambda_{i}\left(d\right)\right)\,.
\end{array}
\]
\end{proof}
\selectlanguage{american}%
\begin{defn}
Define three sets of points $\left(\mu,\nu\right)\in\mathbb{R}^{2}$:
\begin{enumerate}
\item \emph{Eligible points} $\left(\mu,\nu\right)\in E_{i}$ meet $\mu=\lambda_{i}\left(c\right)$
and $\nu=\lambda_{i}\left(d\right)$ for some $\left(c,d\right)\in K_{1}\times K_{2}$,
\item the \emph{upper part} is $U=\left\{ \nu>\mu\right\} $, and
\item the \emph{big points }are those in $B=\left\{ g>1\right\} $.
\end{enumerate}
\end{defn}
\begin{claim}
\label{claim:eligible-margin}The margin $M:=U\cap B$ contains at
least one eligible point.\end{claim}
\selectlanguage{english}%
\begin{thm}
\label{thm:sufficient-by-eligibility}Condition\foreignlanguage{american}{~\ref{condition:sufficient-by-oscillation}
follows from Claim~\ref{claim:eligible-margin}.}\end{thm}
\begin{proof}
By \foreignlanguage{american}{Claim~\ref{claim:eligible-margin}
there is at least one $\left(\mu,\nu\right)\in M\cap E_{i}$. Then
there are indices $c$ and $d$ with $\mu=\lambda_{i}\left(c\right)$,
$\nu=\lambda_{i}\left(d\right)$ because of $\left(\mu,\nu\right)\in E_{i}$
and $\left(\mu,\nu\right)\in U$ implies $\nu>\mu$. Lemma~\ref{lem:factorised-oscillation-quotient}
is invoked to obtain }$f_{i}\left(c,d\right)=g\left(\mu,\nu\right)$
which is greater than 1 because of \foreignlanguage{american}{$\left(\mu,\nu\right)\in B$.}\end{proof}
\begin{defn}
L\foreignlanguage{american}{et}
\selectlanguage{american}%
\begin{enumerate}
\item $\varphi_{0}:=\arctan\left(e^{\epsilon_{\infty}}\right)$ for $\epsilon_{\infty}:=\ln\left(\frac{1+\delta_{i}}{1-\delta_{i}}\right)$,
\selectlanguage{english}%
\item $\epsilon_{\mu,\nu}:=\ln\left(\frac{1+\delta_{i}e^{-b\mu}}{1-\delta_{i}e^{-b\nu}}\right)$
and $\epsilon_{r,\varphi}:=\ln\left(\frac{1+\delta_{i}e^{-br\cos\varphi}}{1-\delta_{i}e^{-br\sin\varphi}}\right)$,
a context-sensitive notation.
\end{enumerate}
\end{defn}
\selectlanguage{american}%
\begin{fact}
\selectlanguage{american}%
If \foreignlanguage{english}{\textup{$\mu=r\cos\varphi$}} and \foreignlanguage{english}{\textup{$\nu=r\sin\varphi$
}}then \foreignlanguage{english}{\textup{$\epsilon_{r,\varphi}=\epsilon_{\mu,\nu}$
}}is positive because $\frac{e^{br\cos\varphi}+\delta_{i}}{e^{br\sin\varphi}-\delta_{i}}>\frac{e^{br\cos\varphi}}{e^{br\sin\varphi}}$.\foreignlanguage{english}{
With $\frac{d}{dr}e^{\epsilon_{r,\varphi}}=-\left(\frac{1}{1-\delta_{i}e^{-b\nu}}b\delta_{i}e^{-b\mu}\cos\varphi+\frac{\left(1+\delta_{i}e^{-b\mu}\right)}{\left(1-\delta_{i}e^{-b\nu}\right)^{2}}b\delta_{i}e^{-b\nu}\sin\varphi\right)<0$
it turns out $\epsilon_{r,\varphi}$ is decreasing in $r$ for all
$\varphi$ such that $\epsilon_{\infty}<\epsilon_{r,\varphi}$ and
$\epsilon_{r,\varphi}\rightarrow\epsilon_{\infty}$ as $r\rightarrow\infty$.
Note \textup{$\epsilon_{\infty}=2\mathrm{artanh}\delta_{i}>2\delta_{i}>0$},
rf. \textup{\cite[§0.7 16.]{Jeffrey:2003}} or \textup{\cite[4.6.22]{Abramowitz:Stegun:1964}}.}

\selectlanguage{english}%
\end{fact}
\begin{lem}
\label{lem:osciallation-quotient-by-cotangent}
\begin{enumerate}
\item $g\left(\mu,\mu\right)=\frac{e^{b\mu}+\delta_{i}}{e^{b\mu}-\delta_{i}}=e^{\epsilon_{r,\varphi}}$
if $r=\mu\sqrt{2}$ and $\varphi=\frac{\pi}{4}$.
\selectlanguage{english}%
\item \textup{If $\mu=r\cos\varphi$ and $\nu=r\sin\varphi$ }and\textup{
}$r\rightarrow\infty$ for constant $\varphi$ then\foreignlanguage{american}{
\[
g\left(\mu,\nu\right)=e^{\epsilon_{\mu,\nu}}\cdot\cot\left(\mathrm{atan2}\left(\mu,\nu\right)\right)\searrow e^{\epsilon_{\infty}}\cdot\cot\varphi>\cot\varphi\,.
\]
}
\selectlanguage{american}%
\item $\nabla g\left(\mu,\nu\right)=\frac{1}{\nu}\cdot\frac{e^{b\nu}}{e^{b\nu}-\delta_{i}}\cdot\left(1+(1-\mu b)\cdot\delta_{i}e^{-b\mu}\;,\;-\mu\cdot\left(\frac{1}{\nu}+\frac{b\cdot\delta_{i}}{e^{b\nu}-\delta_{i}}\right)\cdot\left(1+\delta_{i}e^{-b\mu}\right)\right)$.
\end{enumerate}
\end{lem}
\begin{proof}
For $g\left(\mu,\nu\right)=\frac{\mu}{\nu}\cdot\left(1+\delta_{i}e^{-b\mu}\right)\cdot\frac{e^{b\nu}}{e^{b\nu}-\delta_{i}}$
from Lemma~\ref{lem:factorised-oscillation-quotient} it holds true
that
\begin{enumerate}
\item Both sides are equal to $\frac{\mu}{\mu}\cdot\frac{e^{b\mu}+\delta_{i}}{e^{b\mu}}\cdot\frac{e^{b\mu}}{e^{b\mu}-\delta_{i}}$,
and
\item $g\left(\mu,\nu\right)=\frac{r\cos\varphi}{r\sin\varphi}\cdot\frac{e^{br\cos\varphi}+\delta_{i}}{e^{br\cos\varphi}}\cdot\frac{e^{br\sin\varphi}}{e^{br\sin\varphi}-\delta_{i}}$
such that $e^{\epsilon_{r,\varphi}}e^{br\left(\cos\varphi-\sin\varphi\right)}=\frac{e^{br\cos\varphi}+\delta_{i}}{e^{br\sin\varphi}-\delta_{i}}$
causes
\[
\begin{array}{rl}
g\left(\mu,\nu\right) & =\cot\varphi\cdot e^{br\left(\sin\varphi-\cos\varphi\right)}\cdot\frac{e^{br\cos\varphi}+\delta_{i}}{e^{br\sin\varphi}-\delta_{i}}=\cot\varphi\cdot e^{br\left(\sin\varphi-\cos\varphi\right)}\cdot e^{\epsilon_{r,\varphi}}e^{br\left(\cos\varphi-\sin\varphi\right)}\\
 & =e^{\epsilon_{r,\varphi}}\cdot\cot\varphi\,.
\end{array}
\]

\item $\begin{array}[t]{rl}
\frac{\partial}{\partial\mu}g\left(\mu,\nu\right) & =\frac{1}{\nu}\cdot\frac{e^{b\nu}}{e^{b\nu}-\delta_{i}}\cdot\left(1+\delta_{i}e^{-b\mu}-\mu\cdot b\cdot\delta_{i}e^{-b\mu}\right)\\
 & =\frac{1}{\nu}\cdot\frac{e^{b\nu}}{e^{b\nu}-\delta_{i}}\cdot\left(1+(1-\mu b)\cdot\delta_{i}e^{-b\mu}\right)\qquad\textrm{and}
\end{array}$\\
$\begin{array}[t]{rl}
\frac{\partial}{\partial\nu}g\left(\mu,\nu\right) & =\mu\cdot\left(1+\delta_{i}e^{-b\mu}\right)\cdot\left(-\frac{1}{\nu^{2}}\cdot\frac{e^{b\nu}}{e^{b\nu}-\delta_{i}}+\frac{1}{\nu}\cdot\left(\frac{be^{b\nu}\cdot\left(e^{b\nu}-\delta_{i}\right)-e^{b\nu}\cdot be^{b\nu}}{\left(e^{b\nu}-\delta_{i}\right)^{2}}\right)\right)\\
 & =-\frac{\mu}{\nu}\cdot\left(\frac{1}{\nu}+\frac{b\cdot\delta_{i}}{e^{b\nu}-\delta_{i}}\right)\cdot\frac{e^{b\nu}}{e^{b\nu}-\delta_{i}}\cdot\left(1+\delta_{i}e^{-b\mu}\right)<0\,.
\end{array}$
\end{enumerate}
\end{proof}
\vspace{-2ex}

Figure 4.1: Contour Plot of $g$ and Cartesian Plots of $\Delta\varphi\left(x,\varphi\right)$,
rf. Appendix\hspace*{\fill}\\

\vspace{-4ex}

\begin{figure}[H]
\subfloat[\label{fig:oscillation-quotient}$g$ if $\left(b,\delta_{i}\right)=\left(0.5,1\right)$]{\includegraphics[width=7cm]{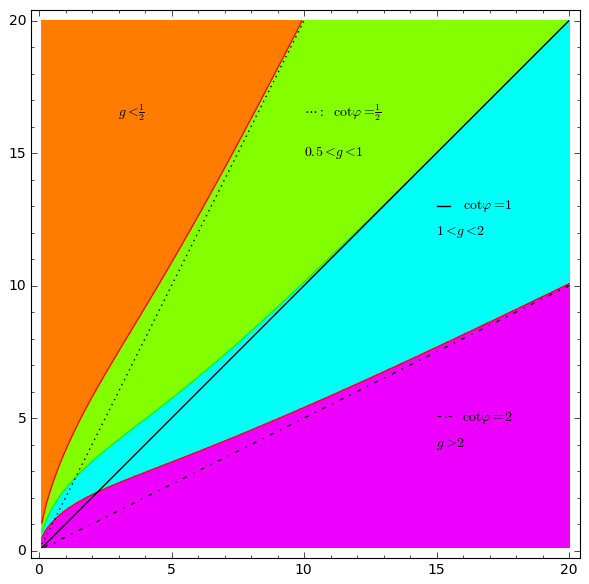}}\subfloat[\label{fig:delta-phi}$\Delta\varphi\left(x,\varphi\right)$ and $\frac{\partial}{\partial\varphi}\Delta\varphi\left(x,\varphi\right)$]{\includegraphics[width=7cm]{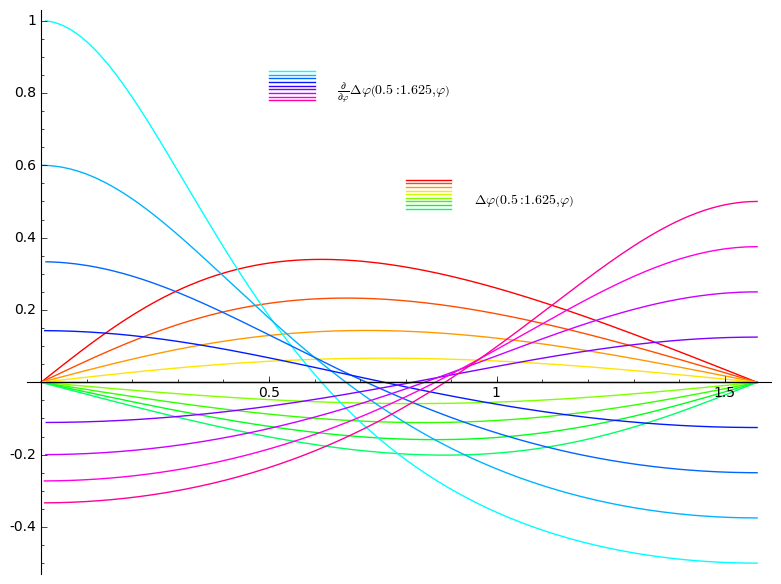}}
\end{figure}

Multiplying $\cot\varphi$ by $e^{\epsilon_{r,\varphi}}$ can be realised
by adding $\Delta\varphi\left(x,\varphi\right):=\mathrm{arccot}\left(x\cdot\cot\varphi\right)-\varphi$
to $\varphi$ presuming $x=e^{\epsilon_{r,\varphi}}>0$ and $\varphi\in\left(0,\frac{\pi}{2}\right)$
s.t. $\Delta\varphi\left(x,\varphi\right)<0$ iff $x>1$. With $\mathrm{arccot}$
$\Delta\varphi\left(x,\varphi\right)$ is decreasing in $x$ and $\varphi\mapsto\frac{\partial}{\partial\varphi}\Delta\varphi\left(x,\varphi\right)=\frac{x\cdot\csc\left(\varphi\right)^{2}}{x^{2}\cot\left(\varphi\right)^{2}+1}-1$
has one change of sign. So $\Delta\varphi\left(x,\varphi\right)$
as a function of $\varphi$ has one minimal and one maximal turning
point for $x>1$ and for $x<1$, resp.
\begin{cor*}
$\cot\left(\varphi+\Delta\varphi\left(x,\varphi\right)\right)=\cot\left(\varphi+\mathrm{arccot}\left(x\cdot\cot\varphi\right)-\varphi\right)=x\cdot\cot\left(\varphi\right)$
and if \foreignlanguage{english}{\textup{$\mu=r\cos\varphi$ and $\nu=r\sin\varphi$}}
then $g\left(\mu,\nu\right)=\cot\left(\varphi+\Delta\varphi\left(e^{\epsilon_{r,\varphi}},\varphi\right)\right)$\foreignlanguage{english}{\textup{.}}\end{cor*}
\selectlanguage{english}%
\begin{prop}
\textup{\label{prop:osciallation-quotient-by-tangent}$g\left(\mu,\nu\right)=1$}
for \textup{$\mu=r\cos\varphi$ }\foreignlanguage{american}{and}\textup{
$\nu=r\sin\varphi$} iff \textup{$\tan\varphi=\frac{1+\delta_{i}e^{-br\cos\varphi}}{1-\delta_{i}e^{-br\sin\varphi}}$}.
Moreover \foreignlanguage{american}{$g\left(\mu,\nu\right)>1$ iff
$\tan\varphi<e^{\epsilon_{r,\varphi}}$ such that }\textup{$e^{\epsilon_{r,\varphi}}>1$}\foreignlanguage{american}{
for $\varphi=\frac{1}{4}\pi$.}\end{prop}
\begin{proof}
$1=g\left(\mu,\nu\right)=\cot\left(\varphi+\Delta\varphi\left(e^{\epsilon_{r,\varphi}},\varphi\right)\right)$
holds if and only if $\frac{\pi}{4}=\mathrm{arccot}\left(e^{\epsilon_{r,\varphi}}\cdot\cot\varphi\right)$
which is true if and impossible unless $\tan\varphi=e^{\epsilon_{r,\varphi}}$.
Likewise $\frac{\pi}{4}>\mathrm{arccot}\left(e^{\epsilon_{r,\varphi}}\cdot\cot\varphi\right)$
holds if and only if \foreignlanguage{american}{$\tan\varphi<e^{\epsilon_{r,\varphi}}$.
A special case is $\tan\varphi=1<\frac{e^{br\sqrt{2}/2}+\delta_{i}}{e^{br\sqrt{2}/2}-\delta_{i}}$
for $\varphi=\frac{1}{4}\pi$.}
\end{proof}
\selectlanguage{american}%

\selectlanguage{english}%
\begin{prop}
\textup{\label{prop:approaching-45-degrees-by-bounded-gaps}}\foreignlanguage{american}{Let
$\left(a_{n}\right)_{n}$ be a sequence in $\mathbb{R}$ with $H<\infty$
for $H:=\liminf\limits _{n\rightarrow\infty}a_{n+1}-a_{n}$ and $a_{n}\rightarrow\infty$.
Then $\arctan\frac{a_{1+n_{k}}}{a_{n_{k}}}\rightarrow\frac{\pi}{4}$
as $k\to\infty$ for the indices $\left(n_{k}\right)_{k}$ of a suitable
subsequence. Moreover $a_{n_{k}}<a_{1+n_{k}}$ and $\arctan\frac{a_{1+n_{k}}}{a_{n_{k}}}>\frac{\pi}{4}$
for all natural $k$.}\end{prop}
\selectlanguage{american}%
\begin{proof}
A sequence $\left(n_{k}\right)_{k}$ of indices with $a_{1+n_{k}}-a_{n_{k}}\rightarrow H$
and $a_{n_{k}}\rightarrow\infty$ can be chosen. $\left(a_{n_{k}},a_{1+n_{k}}\right)\in\mathbb{R}^{2}$
in polar coordinates has the angle $\varphi_{k}$ with $\tan\varphi_{k}=1+\frac{a_{1+n_{k}}-a_{n_{k}}}{a_{n_{k}}}\leq1+\frac{H+\epsilon}{a_{n_{k}}}\rightarrow1$
for an arbitrarily fixed $\epsilon>0$ if $k$ is sufficiently large.
$\left(a_{n_{k}}\right)_{k}$ has infinitely many members with $a_{n_{k}}<a_{1+n_{k}}$
since $a_{n_{k}}\rightarrow\infty$. Other members are not suitable.\end{proof}
\selectlanguage{english}%
\begin{prop}
\textup{\label{prop:entering-thin-slices-near-45-degrees}}\foreignlanguage{american}{If
$\left(a_{n}\right)_{n}$ is an increasing sequence in $\mathbb{R}$
with $\varphi_{n+1}:=\arctan\frac{a_{1+n}}{a_{n}}\rightarrow\frac{\pi}{4}$
as $n\rightarrow\infty$ then there is an index $n$ with $\left(a_{n},a_{1+n}\right)\in M$.}\end{prop}
\selectlanguage{american}%
\begin{proof}
If $\left(a_{n},a_{1+n}\right)\notin M$ was true for all $n$ then
each $n$ would meet either $a_{n}\geq a_{1+n}$ or $g\left(a_{n},a_{1+n}\right)\leq1$.
Therefore $\tan\varphi_{n}\geq e^{\epsilon_{r_{n},\varphi_{n}}}>e^{\epsilon_{\infty}}>1$
follows from Proposition~\ref{prop:osciallation-quotient-by-tangent}
for all $n$ where \foreignlanguage{english}{$a_{n}=r_{n}\sin\varphi_{n}$
and $a_{1+n}=r_{n}\sin\varphi_{n}$. But this contradicts the assumption
}$\tan\varphi_{n}\rightarrow1$ as $n\rightarrow\infty$.\end{proof}
\begin{thm}
\label{thm:bounded-gaps}\cite{Zhang:2013}: $\liminf\limits _{n\rightarrow\infty}p_{n+1}-p_{n}\leq H$
for $H=70\cdot10^{6}$.\end{thm}
\begin{prop*}
Polymath8:

The last theorem holds true with $H=5414$.\end{prop*}
\begin{cor}
\label{cor:primes-in-margin}$M$ contains infinitely many pairs $\left(p_{n_{k}},p_{1+n_{k}}\right)$
of consecutive primes.\end{cor}
\begin{proof}
$\varphi_{k}:=\arctan\frac{p_{1+n_{k}}}{p_{n_{k}}}\rightarrow\frac{\pi}{4}$
with $\varphi_{k}>\frac{\pi}{4}$ and $\frac{p_{1+n_{k}}}{p_{n_{k}}}\rightarrow1$
for a suitable sequence of indices $n_{k}$ follow from Proposition~\ref{prop:approaching-45-degrees-by-bounded-gaps}
because its requirements are fulfilled by Theorem~\ref{thm:bounded-gaps}.
Proposition~\ref{prop:entering-thin-slices-near-45-degrees} shows
that there is a pair $\left(p_{n},p_{1+n}\right)\in M$ for some index
$n$. The argument can be applied to the sequence of primes above
$p_{1+n}$, too.
\end{proof}
Figure~\ref{fig:oscillation-quotient} seems to show that the margin
$M$ is essentially a bulge that only allows eligible points with
small coordinates. But the so-called bulge depends on the choice of
$\delta_{i}$ and disappears as $\delta_{i}$ approaches zero. The
factor $e^{\epsilon_{r,\varphi}}$ has a positive lower limit because
of which the contour $g=1$ diverges away from the bisecting line.
Because of the convergence \foreignlanguage{english}{$\epsilon_{r,\varphi}\searrow\epsilon_{\infty}$}
the asymptote is given by $\mathrm{atan2}\left(\mu,\nu\right)=\frac{\pi}{4}+\Delta\varphi\left(e^{-\epsilon_{\infty}},\frac{\pi}{4}\right)>\frac{\pi}{4}$.
\begin{lem}
\label{lem:asymptotically-small-loglog-quotient}$\frac{\ln\ln nx}{\ln\ln n}\rightarrow1$
as $n\rightarrow\infty$ for every fixed real number $x$.\end{lem}
\begin{proof}
Put $a_{n}:=\ln n+\ln x$. Since $\ln\ln nx\rightarrow\infty\leftarrow\ln\ln n$
as $n\rightarrow\infty$ l'Hôspital's rule implies
\[
\frac{\ln\ln nx}{\ln\ln n}\sim\frac{n\ln n}{n\ln nx}=1-\frac{\ln x}{a_{n}}\rightarrow1\quad(n\rightarrow\infty)\,.
\]
\end{proof}
\begin{cor}
\label{cor:asymptotically-small-loglog-quotient}If Assumption~\ref{assumption:Alaoglu-Erdoes-conjecture}
is true then $\liminf\limits _{c\rightarrow\infty}\mathcal{L}_{i+c,1}=1$
and $\lambda_{i}\left(k\right)\rightarrow\infty$ as $k\rightarrow\infty$.\end{cor}
\begin{proof}
Requiring $\varepsilon_{i+c_{k}+1}=F\left(p,v\right)$ for a fixed
prime $p$ and $v>v_{p}\left(n_{i}\right)$ determines a sequence
$\left(c_{k}\right)_{k}$ such that $q_{i+c_{k}+1}=p$ for all $k$.
Then $\liminf\limits _{k\rightarrow\infty}\mathcal{L}_{i+c_{k},1}=1\geq\liminf\limits _{c\rightarrow\infty}\mathcal{L}_{i+c,1}\geq1$
follows with Lemma~\ref{lem:asymptotically-small-loglog-quotient}.
\foreignlanguage{english}{$\lambda_{i}\left(k\right)=\ln\ln n_{i+k}\rightarrow\infty$
as $k\rightarrow\infty$ follows from }Theorem~\ref{thm:Gr=0000F6nwall}.\end{proof}
\begin{prop}
\label{prop:approaching-45-degrees-by-bounded-quotients}Let $\left(a_{n}\right)_{n}$
be an increasing sequence in $\mathbb{R}$ with $\liminf\limits _{n\rightarrow\infty}\frac{a_{n+1}}{a_{n}}=1$.
Then $\arctan\frac{a_{1+n_{k}}}{a_{n_{k}}}\rightarrow\frac{\pi}{4}$
as $k\to\infty$ for the indices $\left(n_{k}\right)_{k}$ of a suitable
subsequence. Moreover $a_{n_{k}}<a_{1+n_{k}}$ and $\arctan\frac{a_{1+n_{k}}}{a_{n_{k}}}>\frac{\pi}{4}$
for all natural $k$.\end{prop}
\begin{proof}
A sequence $\left(n_{k}\right)_{k}$ of indices with $\frac{a_{n+1}}{a_{n}}\rightarrow1$
can be chosen. $a_{n_{k}}<a_{1+n_{k}}$ holds for infinitely many
members of $\left(a_{n_{k}}\right)_{k}$ because $a_{n}\rightarrow\infty$.
In polar coordinates the points $\left(a_{n_{k}},a_{1+n_{k}}\right)\in\mathbb{R}^{2}$
have the angle $\varphi_{k}$ with $\tan\varphi_{k}=\frac{a_{1+n_{k}}}{a_{n_{k}}}\leq1+\epsilon$
for an arbitrarily fixed $\epsilon>0$ if $k$ is sufficiently large.
Members with $a_{n_{k}}\geq a_{1+n_{k}}$ are not suitable.\end{proof}
\begin{cor}
\label{cor:eligible-margin}$M$ contains eligible points.\end{cor}
\begin{proof}
$\varphi_{k}:=\arctan\frac{\lambda_{i}\left(1+n_{k}\right)}{\lambda_{i}\left(n_{k}\right)}\rightarrow\frac{\pi}{4}$
as $k\rightarrow\infty$ is true with $\varphi_{k}>\frac{\pi}{4}$
for a suitable sequence of indices $n_{k}$ as Proposition~\ref{prop:approaching-45-degrees-by-bounded-quotients}
can be applied because of Corollary~\ref{cor:asymptotically-small-loglog-quotient}.
Proposition~\ref{prop:entering-thin-slices-near-45-degrees} shows
that there is a pair $\left(\lambda_{i}\left(k\right),\lambda_{i}\left(1+k\right)\right)\in M$
for some index $k$. It should also be possible to apply the argument
to the sequence of CA numbers above $n_{i+k}$ but one eligible point
is sufficient.\end{proof}
\begin{conclusion}
\label{conclusion:reduction-of-RIE}Claim~\ref{claim:general-validity-of-RIE}
follows from Assumption~\ref{assumption:Alaoglu-Erdoes-conjecture}.\end{conclusion}
\begin{proof}
A chain of reductions:
\begin{enumerate}
\item Claim~\ref{claim:general-validity-of-RIE} is reduced to Condition~\ref{condition:sufficient}
by Theorem~\ref{thm:sufficient},
\item which is reduced to Condition~\ref{condition:separated} by Theorem~\ref{thm:extended-Robins-Method},
\item which is reduced to Condition~\ref{condition:sufficient-by-oscillation}.
The contradiction to Assumption~\ref{assumption:the-least-exceptional-number}
in Theorem~\ref{thm:sufficient-by-oscillation} is achieved with
Assumption~\ref{assumption:Alaoglu-Erdoes-conjecture} and the methods
of section~\ref{sec:Subsequent Maximisers}.
\item Condition~\ref{condition:sufficient-by-oscillation} is reduced to
Claim~\ref{claim:eligible-margin} by Theorem~\ref{thm:sufficient-by-eligibility}
and
\item Claim~\ref{claim:eligible-margin} is established by Corollary~\ref{cor:eligible-margin}
for which Assumption~\ref{assumption:Alaoglu-Erdoes-conjecture}
has been used in Corollary~\ref{cor:asymptotically-small-loglog-quotient},
too.
\end{enumerate}
\end{proof}

\section{Final Remarks\label{sec:Final-Remarks}}

Recently it has already been pointed out in \cite{Cheng:Albeverio:2012,Cheng:Albeverio:2013,Cheng:2013}
that RIE holds for all $n>5040$ without requiring Assumption~\ref{assumption:Alaoglu-Erdoes-conjecture}.
\foreignlanguage{english}{Independently of this approach }Conclusion~\ref{conclusion:reduction-of-RIE}\foreignlanguage{english}{
will have many consequences once }Assumption~\ref{assumption:Alaoglu-Erdoes-conjecture}\foreignlanguage{english}{
is established. A few of them are metioned.}
\selectlanguage{english}%
\begin{enumerate}
\item RH follows with the original papers \cite{Robin:1983:ordre:maximum,Robin:1984:grandes:valeurs}
while GRH remains undecided.
\item The weakened version $M(x)=O(x^{\frac{1}{2}+\varepsilon})$ of the
disproved Mertens conjecture.
\item The only extraordinary number is 4, \cite{Caveney:2011,Caveney:2012}.
\item There are infinitely many extremely abundant numbers, \cite[Thm 2.4]{Nazardonyavi:Yakubovich:2013}.
\item The status of Cramér's conjecture is still undecided but with Cramér's
work $O\left(\sqrt{p}\cdot\ln p\right)$ can be deduced for every
gap.
\item A recent result is Hypothesis P in \cite{deReyna:vandeLune:2013}
according to Proposition 40 in that paper.
\item The Riesz criterion, Nicolas' inequality for $\varphi$, Weil's and
Li's criterions, and Speiser's statement on $\zeta'$, \cite{Riesz:1916,Nicolas:1983,Banks:Hart:Moree:Nevans:2009,Weil:1952,LiXianJin:1997,Speiser:1934}.
\end{enumerate}
\selectlanguage{american}%
Approaches related to the present one are \foreignlanguage{english}{\cite{Nazardonyavi:Yakubovich:2013,Nazardonyavi:Yakubovich:2013:Delicacy}
as well as \cite{Morkotun:2013,Akbary:Friggstad:2009}. The former
led to }\cite[Thm 1.7]{Nazardonyavi:Yakubovich:2013:Delicacy}\foreignlanguage{english}{
and a sequence of increasing values of $X$ whose existence follows
from }Grönwall's theorem with Robin's Oscillation theorems.\foreignlanguage{english}{
The latter pointed out that increasing values of $\frac{\sigma\left(n\right)}{n\cdot\ln\ln n}$
on superabundant numbers are sufficient. CA numbers do not allow for
the minimality condition. Exceptional numbers cause oscillations whereas
the explicit formulas are more precise under RH. If oscillations prevent
exceptional numbers RH could be said to be hoist by its own petard.}

\section*{Acknowledgements}

Unfortunately I never received the support I would have liked to thank
for at this place for which there is a variety of reasons. However,
I want to thank Dr. Thomas Severin for recommending to me to investigate
RIE when I worked for the Allianz insurance company in 1997. Likewise
I thank Alexander Rueff for his suggestion to study Ramanujan's lost
notebook during my first year at university. Thanks for helpful feedback
from an anonymous referee. I appreciate many fruitful discussions
with Sebastian Spang while I worked with him, again for the Allianz.
I thank Keith Briggs for our conversation, too. Last not least I thank
my cousin Benedict Scholl for an additional pair of eyes as it might
have been difficult for him to provide the strictly non-mathematical
review I asked for.

\appendix

\section{\label{sec:Implementation}Implementation}

Results have already been mentioned above. Sage code follows below.
My first version computed for two weeks last year on a MacBook Air
until $X\left(n\right)>1.781$, i.e. $n_{143215}$ was reached. A
revised implementation did the job in a bit more than half an hour
(without standby phases). There are two reasons for the difference:
\begin{enumerate}
\item Pre-Computation of a list of primes,
\item Consequent exploitation of the factorisation of CA numbers.
\end{enumerate}
Because of assumption~\ref{assumption:Alaoglu-Erdoes-conjecture}
it is only necessary to select \textbf{the} next prime in every loop,
determine new primes that may follow next, and to compute the values
of $\varepsilon$ associated with the (in virtue of \cite[§59]{Ramanujan:1997}
and \cite[Thm 1]{Alaoglu:Erdoes:1944} at most 2) additional new primes.

The following functions compute CA numbers as they were represented
in section~\ref{sub:Number-Crunching}. Noe's top-down form of primes
triggering the next valuation is used to store $n_{i}$. When $n_{i}$
has been computed the primes $p$ such that $n_{i}p$ does not violate
the \emph{basic SA condition} are called \emph{candidates for $q_{i+1}$},
rf. \cite[Theorem 2]{Alaoglu:Erdoes:1944}. It is convenient that
the candidates for $q_{i+1}$ are the primes that occur in the bottom-up
form of $n_{i}$.

TODO: in \texttt{triggers} and \texttt{candidates} store indices in
\texttt{sieve} instead of elements of \texttt{sieve}, endow \texttt{sieve}
with logs of primes, in \texttt{addSievedPrimeToTriggers()} avoid
searching \texttt{newprime} in \texttt{sieve}.
\begin{algorithm}
Compute CA Numbers in top-down form, a potentially not so big CA number
is to be given, e.g. counter = 4 and triggers = {[}5, 2{]}, seems
not to work with the known smaller CA numbers
\begin{lyxcode}
sieve~=~prime\_range(2,~50000000)

def~getSubsequentCAnumber(counter,~triggers,~number,~sieve):

~~~~k~=~0

~~~~candidates~=~getBottomUpTriggers(triggers)

~~~~epsilons~=~map(lambda~i:~getCAparameter(candidates{[}i{]},~i+1),

~~~~~~~~~~~~~~~~~~~range(len(candidates)))

~~~~while~k~<~number:

~~~~~~~~k~=~k~+~1

~~~~~~~~vmax~=~selectNextTrigger(epsilons)

~~~~~~~~triggers~=~addSievedPrimeToTriggers(candidates{[}vmax{]},~triggers,

~~~~~~~~~~~~~~~~~~~~~~~~~~~~~~~~~~~~~~~~~~~~candidates,~epsilons,~sieve)

~~~~return~triggers
\end{lyxcode}
\end{algorithm}

\begin{algorithm}
Convert CA Numbers to bottom-up form
\begin{lyxcode}
def~getBottomUpTriggers(triggers):

~~~~l~=~len(triggers);~i~=~0

~~~~result~=~{[}{]}

~~~~while~i<l:

~~~~~~~~p~=~triggers{[}i{]};~j~=~1

~~~~~~~~if~p!=0:

~~~~~~~~~~~~result.append(next\_prime(p))

~~~~~~~~else:

~~~~~~~~~~~~p~=~triggers{[}i+j{]}

~~~~~~~~~~~~while~p==0:

~~~~~~~~~~~~~~~~j=j+1

~~~~~~~~~~~~~~~~p~=~triggers{[}i+j{]}

~~~~~~~~~~~~result.append(next\_prime(p))

~~~~~~~~~~~~result.extend({[}0~for~dummy~in~range(j){]})

~~~~~~~~~~~~j=j+1

~~~~~~~~i=i+j

~~~~result.append(2)

~~~~return~result
\end{lyxcode}
\end{algorithm}

\begin{algorithm}
Choose the next candidate
\begin{lyxcode}
def~selectNextTrigger(E):

~~~~emax~=~max(E)

~~~~vmax~=~{[}v~for~v~in~range(len(E))~if~E{[}v{]}~>=~emax{]}

~~~~if~len(vmax)~>~1:

~~~~~~~~print~\textquotedbl{}FOUR~EXPONENTIALS~DISPROVED!\textquotedbl{}

~~~~else:

~~~~~~~~return~vmax{[}0{]}
\end{lyxcode}
\end{algorithm}

\begin{algorithm}
Enter the next candidate in the list of triggering primes and update
candidates and epsilons
\begin{lyxcode}
def~addSievedPrimeToTriggers(newprime,~triggers,~candidates,~epsilons,~sieve):

~~~~vmax~=~candidates.index(newprime)+1

~~~~npi~=~sieve.index(newprime)

~~~~i~=~vmax-1

~~~~l~=~len(triggers)

~~~~if~i~<~l:

~~~~~~~~triggers{[}i{]}~=~newprime

~~~~~~~~if~candidates{[}i+1{]}~==~0:

~~~~~~~~~~~~candidates{[}i+1{]}~=~newprime

~~~~~~~~~~~~epsilons{[}i+1{]}~=~getCAparameter(newprime,~vmax+1)

~~~~~~~~if~triggers{[}i-1{]}~==~newprime:

~~~~~~~~~~~~triggers{[}i-1{]}~=~0

~~~~~~~~~~~~candidates{[}i{]}~=~0

~~~~~~~~~~~~epsilons{[}i{]}~=~0

~~~~~~~~else:

~~~~~~~~~~~~candidates{[}i{]}~=~sieve{[}npi+1{]}

~~~~~~~~~~~~epsilons{[}i{]}~=~getCAparameter(candidates{[}i{]},~vmax)

~~~~else:

~~~~~~~~triggers.append(newprime)

~~~~~~~~candidates.append(newprime)

~~~~~~~~vmax~=~vmax~+~1

~~~~~~~~epsilons.append(getCAparameter(newprime,~vmax))

~~~~~~~~if~triggers{[}i-1{]}~>~0:

~~~~~~~~~~~~triggers{[}i-1{]}~=~0

~~~~~~~~~~~~candidates{[}i{]}~=~0

~~~~~~~~~~~~epsilons{[}i{]}~=~0

~~~~return~triggers
\end{lyxcode}
\end{algorithm}

\begin{algorithm}
Plotting $\Delta\varphi\left(x,\:\varphi\right)$ and $\frac{\partial}{\partial\varphi}\Delta\varphi\left(x,\:\varphi\right)$
in Figure~\ref{fig:delta-phi}.
\begin{lyxcode}
DeltaPhi(x,~phi)~=~arccot(x{*}cot(phi))-phi

P~=~sum({[}plot(DeltaPhi(i/8,phi),~(phi,~0,~pi/2),

~~~~~~~~~~~~~~~~~~rgbcolor~=~hue(((i+16)\%20)/20))~for~i~in~range(4,~13){]})

P~=~P~+~sum({[}line({[}(0.8,~0.6-i/100),~(0.9,~0.6-i/100){]},

~~~~~~~~~~~~~~~~~~rgbcolor~=~hue(((i+16)\%20)/20))~for~i~in~range(4,~13){]})

P~=~P~+~text('\$\textbackslash{}\textbackslash{}Delta\textbackslash{}\textbackslash{}varphi\textbackslash{}\textbackslash{}left(0.5:1.625,\textbackslash{}\textbackslash{}varphi\textbackslash{}\textbackslash{}right)\$',

~~~~~~~~~~~~~~~~~~(0.95,~0.5),~color=\textquotedbl{}black\textquotedbl{},~horizontal\_alignment='left')

P~=~P~+~sum({[}plot(DeltaPhi.diff(phi)(i/8,phi),~(phi,~0,~pi/2),

~~~~~~~~~~~~~~~~~~rgbcolor~=~hue(((i+6)\%20)/20))~for~i~in~range(4,~13){]})

P~=~P~+~sum({[}line({[}(0.5,~0.9-i/100),~(0.6,~0.9-i/100){]},

~~~~~~~~~~~~~~~~~~rgbcolor~=~hue(((i+6)\%20)/20))~for~i~in~range(4,~13){]})

(P+text('\$\textbackslash{}\textbackslash{}frac\{\textbackslash{}\textbackslash{}partial\}\{\textbackslash{}\textbackslash{}partial\textbackslash{}\textbackslash{}varphi\}

~~~~~~~~~~~~~~\textbackslash{}\textbackslash{}Delta\textbackslash{}\textbackslash{}varphi\textbackslash{}\textbackslash{}left(0.5:1.625,\textbackslash{}\textbackslash{}varphi\textbackslash{}\textbackslash{}right)\$',

~~~~~~~~~(0.65,~0.8),~color=\textquotedbl{}black\textquotedbl{},~horizontal\_alignment='left')).show()
\end{lyxcode}
\end{algorithm}

\begin{algorithm}
Compute $\varepsilon=F\left(x,v\right)$
\begin{lyxcode}
def~getCAparameter(x,~v):

~~~~if~x~==~0:

~~~~~~~~return~0

~~~~else:

~~~~~~~~if~v~==~0:

~~~~~~~~~~~~return~log(1+1/x)/log(x)

~~~~~~~~else:

~~~~~~~~~~~~return~log((1-x\textasciicircum{}(v+1))/(x-x\textasciicircum{}(v+1)))~/~log(x)
\end{lyxcode}
\end{algorithm}
\pagebreak{}

\bibliographystyle{plain}
\addcontentsline{toc}{section}{\refname}\bibliography{/Users/tsw/Documents/Mathematik/Literatur/Datenbank/Literaturdatenbank}

\end{document}